\documentclass[11pt]{article}

\usepackage{comment}

\usepackage{subcaption}
\usepackage{tikz-cd}
\usepackage{tikz}
\usepackage{quiver}
\usepackage{mathtools}
\usepackage{url}
\usepackage{stmaryrd}
\usepackage{comment}

\usepackage{subcaption}
\usepackage{amsmath}

\usepackage{bussproofs}\EnableBpAbbreviations

\usepackage{cite}
\usepackage{amsmath,amssymb,amsfonts,amsthm}
\usepackage{graphicx}
\usepackage{textcomp}
\usepackage{xcolor}
\def\BibTeX{{\rm B\kern-.05em{\sc i\kern-.025em b}\kern-.08em
    T\kern-.1667em\lower.7ex\hbox{E}\kern-.125emX}}

\usepackage{enumerate}
\usepackage{url}
\usepackage{hyperref}
\usepackage{bm}

\usepackage{algorithm}
\usepackage{algpseudocode}
\usepackage{titlesec}

\usepackage{fullpage}

\usepackage{bm}
\usepackage{tikz}
\usetikzlibrary{automata,positioning}

\newcommand{\stlc}{\mathrm{ST\lambda C}}
\newcommand{\stdlc}{\mathrm{ST\partial\lambda C}}
\newcommand{\set}[1]{\{ #1 \}}
\newcommand{\tansp}[2]{T_{#1}{#2}} 
\newcommand{\cotansp}[2]{T^*_{#1}{#2}} 
\renewcommand{\tan}[1]{T{#1}} 
\newcommand{\cotan}[1]{T^*{#1}} 
\newcommand{\plb}[2]{{#1}^*\!{#2}} 
\newcommand{\cC}{\mathcal{C}}
\newcommand{\lto}{\multimap} 
\newcommand{\R}[1]{\mathbb{R}^{#1}} 
\newcommand{\id}{\mathrm{id}}
\newcommand{\ihom}[2]{[{#1}\to {#2}]} 
\newcommand{\ilhom}[2]{[{#1}\lto {#2}]} 
\newcommand{\curry}{\lambda}
\newcommand{\uncurry}{\lambda^{-1}}
\newcommand{\dual}[1]{{#1}^{\perp}}
\newcommand{\ctr}{c} 
\newcommand{\coctr}{\overline{\ctr}} 
\newcommand{\der}{d} 
\newcommand{\coder}{\overline{\der}} 
\newcommand{\prom}[1]{{{#1}^{!}}} 
\newcommand{\wkn}{w} 
\newcommand{\cowkn}{\overline{\wkn}} 
\newcommand{\pshfwd}[1]{{{#1}_*}} 
\newcommand{\wdial}[1]{{#1}^\bullet} 
\newcommand{\cdial}[2]{{#1}_{#2}} 
\newcommand{\ev}{\mathrm{ev}} 

\newcommand{\Pdial}{\mathbf{P}} 
\newcommand{\vdashPdial}{\vdash_\Pdial} 
\newcommand{\plus}{+
} 
\newcommand{\pmon}[1]{\mathcal{M}[{#1}]} 
\newcommand{\bind}{\texttt{>\!\!>\!=}} 

\newcommand\restr[2]{{
  \left.\kern-\nulldelimiterspace 
  #1 
  \vphantom{\big|} 
  \right|_{#2} 
  }}

\titleformat*{\section}{\large\bfseries}
\titleformat*{\subsection}{\normalsize\bfseries}
\titleformat*{\subsubsection}{\normalsize\em}

\newtheorem{theorem}{Theorem}[section]
\newtheorem{proposition}[theorem]{Proposition}
\newtheorem{lemma}[theorem]{Lemma}
\newtheorem{corollary}[theorem]{Corollary}
\newtheorem{note}[theorem]{Note}

\theoremstyle{remark}
\newtheorem{definition}[theorem]{Definition}
\newtheorem{remark}[theorem]{Remark}

\title{
An excursion into Dialectica and Differentiation}

\author{Davide Barbarossa\\
Department of Computer Science, University of Bath, UK\\
\url{https://davidebarbarossa12.github.io/}\\ \url{db2437@bath.ac.uk}}

\begin{document}

\maketitle

\begin{abstract}
G\"odel's Dialectica has been introduced and developed 
in the tradition of the so-called functional interpretations.
Only recently has it been related with the \emph{a priori} unrelated notion of differentiation, by taking a program-theoretic approach.
We revisit the deep connection between these two notions in order to understand its structural reasons, as well as to express it in an arguably more natural way by following a geometric intuition.
More specifically, we give a logical relation between a Dialectica transformed term and its reverse differential in a differential category and, then, we phrase the Dialectica program transformation in the language of lenses, often used indeed in Automatic Differentiation in order to model reverse differentiation. 
We illustrate how this clarifies why Dialectica behaves as a differentiable program transformation, and what the limits of this correspondence are.
\end{abstract}

\noindent\textbf{Funding.} This work was funded by the EPSRC, grant number EP/W035847/1. For the purpose of Open Access the authors have applied a CC BY public copyright licence to any Author Accepted Manuscript version arising from this submission.\\

\noindent\textbf{Acknowledgments.} We thank Thomas Powell for many instructive discussions about Dialectica.


\section{Introduction}

\paragraph*{Forward differentiation and its categorical formulation}

In elementary calculus, the derivative of a function on the reals is a certain limit, and its differential (giving the error in output for a certain error in input) is the product of the derivative at a point times the input error.
Using partial derivatives, one generalises this to differentiable maps $f:\mathbf{R}^n\to \mathbf{R}^m$, for which the differential $Df:\mathbf{R}^n \times \mathbf{R}^n \to \mathbf{R}^m$, linear in (say) the second argument, is the directional derivative $Df(a,v):=J_a f\cdot v$ ($J_a f$ being the Jacobian of $f$ at $a$).
Considered as a  function of $v$ (thus, linear), we obtain what is called the pushforward $\pshfwd f a$ of $f$ at $a\in\R n$.
This situation can be abstracted by \emph{Cartesian differential categories} \cite{CDCrevisited}, where $D$ is an operator on homsets, and the Cartesian \emph{closed} version admits the simply typed \emph{differential $\lambda$-calculus} ($\stdlc$ for short) \cite{the_diff_lam_calc} as internal language \cite{models_of_dlc_revisited,gallagher_phd,what_is_model_stdlc}.
Thinking of usual smooth maps between smooth manifolds (cfr.\ e.g.\ \cite{diff_geo_book}) as the natural setting for differentiation, the same constructions can be carried out by working in a local chart of coordinates: one constructs the tangent bundle $\tan A:=\sum_{a\in A} \tansp{a}{A}$ of a manifold $A$ and the pushforward of $f:A\to B$ as a linear map between tangent spaces, which one can think of as a dependent function $\pshfwd f : \prod_{a\in A}[\tansp a A \lto \tansp{f(a)}{B}]$ ($\tansp c C$ is the tangent space of $C$ at $c$, and we borrow the notation $\lto$ from Linear Logic for linear functions).
It is standard to see now the differential as the map $Tf: (a,v)\in \tan A \to (f a,\pshfwd f a v)\in \tan B$.
Also this situation can be abstracted categorically, via the so-called \emph{tangent categories} \cite{Diff_struc_Tangent_struc}, where $T$ is a given endofunctor.
This is \emph{forward} differentiation, as the functoriality of $T$ amounts to the usual forward mode for computing the chain rule.
The case of real valued functions, and its closed version of $\stdlc$, is the one of spaces (typically, Euclidean) where the tangent spaces are isomorphic to the base space and thus the tangent bundle trivializes: $\tan A\simeq A\times A$.

\paragraph*{Reverse differentiation and its categorical formulation}

In the realm of manifolds, the component $\pshfwd f a:\tansp{a}{A}\lto \tansp{fa}{B}$ of a pushforward $\pshfwd{f}$ at $a$ is linear, so we can take its dual  map $\dual{(\pshfwd f a)}:\cotansp{fa}{B} \lto \cotansp{a}{A}$, defined in local coordinates by the transpose of the Jacobian.
It is standard (\cite{hatcher_k_theory}) to consider the pullback $\plb{f}{\cotan B}:=\sum_{a\in A} \cotansp{fa}{B}$ of the cotangent bundle (\cite{diff_geo_book}) $\cotan B:=\sum_{b\in B} \cotansp{b}{B}$ along $f$, and let the \emph{reverse differential} of $f$ be $Rf:(a,v)\in \plb{f}{\cotan B}\to (a,\dual{(\pshfwd f a)}v)\in \cotan A$. 
Remarking that $\dual T A:=\prod_{a\in A} \cotansp{a}{A}$ is precisely the space of differential $1$-forms, yet another standard way of expressing $R$ is to see it as a map $\dual T f:\dual T B\to \dual T A$ of differential $1$-forms on $B$ to ones on $A$.
In this way $\dual T$ is \emph{contravariant}: $\dual T(f\circ g)=\dual T g\circ \dual T f$, and this corresponds to the usual backward mode for computing the chain rule. 
We will however stick to $R$ instead of $\dual T$.
Reverse differentiation is abstracted via \emph{reverse tangent categories} \cite{ReverseTC24}, where we do not directly have a functor $\dual T$ or $R$, but rather we have $T$ together with an involution ${(\_)}^*$ on 
fibrations that allows 
one to build $R$.
In the special case where $\cotan A \simeq \tan A \simeq A\times A$, the reverse differential of $f$ is $Rf:\mathbf{R}^n \times \mathbf{R}^m \to \mathbf{R}^n$ (note the swap of $n$ and $m$ w.r.t.\ $Df$), linear in the second argument.
This situation is abstracted by Cartesian \emph{reverse differential categories} \cite{Reverse_derivative_categories,monoidal_rev_diff_cat}, a particular case of reverse tangent ones, where $R$ is a given operator on homsets.

\paragraph*{Dialectica and its program-theoretic formulation}

In \cite{godel_dial} G\"odel defined a transformation $A\mapsto A_D$, known as ``Dialectica'', from intuitionistic arithmetic formulas (say, HA) to System T ones (cfr.\ \cite{Dialectica_by_Avigad_Feferman}).
The transformed formula $A_D$ contains two additional kinds of free variables, playing the role of witnesses ($w$) and counterexamples ($c$) for $A$, so $A_D=A_D\{w,c\}$.
The main theorem is that provability in the source system yields provability in the target, thus obtaining a relative consistency result with respect to the finitist methods of T: if $\vdash_{\mathrm{HA}} A$ then there are terms $M\in T$ witnessing $A$ in the sense that $\vdash_{\mathrm{T}} A_D\{M,c\}$ (for a free counter variable $c$).
Dialectica has been later broadly applied in the field of proof-mining \cite{proofminingbook}, with major developments thereafter.
More recently, on an orthogonal direction, by taking a proof/program-theoretical point of view, i.e.\ looking at Dialectica as a transformation of proofs/programs (and not just a transformation of formulas preserving provability), in \cite{Pedrot14,pedrot_thesis} it is shown that it is also a genuine high-order program transformation.
In particular, if we just consider the simply typed $\lambda$-calculus ($\stlc$ for short), one can define a variant of System T, let us call it $\Pdial$, and see Dialectica as a transformation $\wdial{(\_)}$, with auxiliary transformations $\cdial{(\_)}{x}$ for each variable $x$, of 
programs from $\stlc$ to $\Pdial$.
Even if this restriction involves a definitely logically poor source system (not suited for, say, proof mining), $\stlc$ handles the arrow type, whose treatment is the characteristic feature of G\"odel's Dialectica, and it is of course interesting from a programming viewpoint because it represents high-order computation.
The target $\Pdial$ has simple types with products, plus a monadic type constructor $\mathcal M[\_]$ for multisets (needed, as in the Diller-Nahm 
Dialectica \cite{Diller_Nahm_Dialectica,unif_fun_interpr_Oliva06}, for memorizing the counters in a contraction rule), together with two functions $W$ (witness) and $C$ (counter) from simple types to $\Pdial$-types.
G\"odel's relative consistency theorem becomes then a soundness theorem for the transformations: if $x:A\vdash_{\stlc} M:B$ then $x:W(A)\vdash_{\Pdial} \wdial M: W(B)$ and $x:W(A)\vdash_{\Pdial} \cdial M x:C(B)\to \mathcal M[C(A)]$.
This is reminiscent of the much older sequential algorithms \cite{berry_curien82}, where we have a domain map $D(A)\to D(B)$ and, for $x\in D(A)$ and $y'$ in an ``accessible cell'' in $B$, one accessible $y$ in $A$.

\paragraph*{Dialectica and its links with differentiation}

The above considerations have several interesting consequences, e.g.\ understanding Dialectica as a delimited continuation mechanism.
But the one of interest for this paper is that it allowed to notice, in \cite{KerjeanPedrot24}, that the types of $\wdial M$ and $\cdial M x$ perfectly match the ones of a function and its reverse differential at a point $x$.
This is not a coincidence, as it is shown that Dialectica is indeed a differentiable program transformation: given $M$ a simply typed $\lambda$-term, $\wdial M$ computes the differential of $M$, implementing a reverse differentiation algorithm (i.e.\ by computing the chain rule as a reverse differential -- this is precisely where the $\cdial{(\_)}{x}$ appears).
The authors show this in several ways, most notably by defining the logical relation $\sim$ (and an auxiliary one, $\bowtie$) in \cite[Def.\ 4.8]{KerjeanPedrot24} between the Dialectica transformed programs in $\Pdial$ and the already mentioned $\stdlc$, which is the syntactic way of handling (forward) differentiation in $\lambda$-calculus. 
The main theorem $\sim$-relates then the image of the transformation of a ordinary $\lambda$-term with a certain $\stdlc$-term, which the trained differential $\lambda$-calculist understands as encoding the reverse differentiation of the first.

\paragraph*{This paper}

In this paper we try to both \emph{clarify} and \emph{understand why} Dialectica is related with (reverse) differentiation.
In particular, we will see that it is because of its structural properties, and this makes the link between the two notions arguably more natural to see.
The main claim of this paper is then the following:
\emph{Dialectica is \emph{not} reverse differentiation ``by itself'', but it does indeed \emph{behave} as a reverse differentiation transformation, in the sense that it composes as such, because Dialectica can be expressed via lenses which, in turn, are the abstract compositional setting that models the reverse chain rule.}
We argue for it in a progressive way:

Since the connection with differentiation relies on the $\stlc$, and because we are interested in the structural aspect of Dialectica, for us the latter will mean the one restricted to $\stlc$, and written in the style of \cite{Pedrot14} as the pair of transformations $\wdial{(\_)},\cdial{(\_)}{x}:\stlc\to\Pdial$.

In Section \ref{sec:log_rel}, we propose an alternative presentation of the connection between Dialectica and reverse differentiation. We define the analogue of the already mentioned logical relations $\sim$, $\bowtie$, but now 
instead of relating a Dialectica transformed $\Pdial$-term with a $\stdlc$-term, we relate it with an arrow in an opportune class of differential categories (Definition \ref{def:log_rels}).
The main point is that we use the natural concepts from differential geometry of pushforward and its pullback.
We then show that, on the image of the transformation, such relations perform indeed reverse differentiation (Theorem \ref{th:main}, Corollary \ref{cor:dial_rev}), which is now expressed in the standard clearer way as explained earlier. 
We finally quickly discuss how it seems that one \emph{cannot} immediately lift this relationship to richer logical systems (Figure \ref{fig:tom}), at least for a natural reading of Dialectica.
The categorical setting of this section is motivated by the fact that the $\stdlc$ encodes \emph{forward} differentiation (the $\stdlc$-term $\lambda v. \, D[\lambda x.M,v]N$ is the pushforward of $M$, as a function of $x$, at $N$
); but since we know that Dialectica performs a \emph{reverse} differentiation, we would like to express it in a setting in which we can explicitly talk about the reverse differential, instead of hiding this reverse operation in the syntactic definition of $\sim$, $\bowtie$.
While the 
\emph{closed} version of Cartesian reverse differential categories does not exist in the literature yet, 
we can still use the well-known Cartesian closed differential categories, equipped with enough structure in order to express reverse differentials.

In the next Section \ref{sec:diff_lenses}, we introduce the category of lenses and recall how they precisely express the compositional aspect of the reverse chain rule (Proposition \ref{prop:difcat_to_lenses}, \ref{prop:revdifcat_to_lenses}, \ref{prop:revtangcat_to_lenses}).

This allows us, in Section \ref{sec:dial_is_lens}, to understand that the \emph{structural} aspect of Dialectica can be understood as a transformation into lenses, and that this fact is the responsible for its behaviour as reversed differentiation.
We do this by, first, factorizing the functor in \cite[Proposition 5.7]{KerjeanPedrot24} by remarking that it only uses the subcategory of a Dialectica category only involving trivial subobjects (Corollary \ref{cor:no_subsets}, Proposition \ref{prop:Giso}); 
then, we show how the categorical notion of lens shapes Dialectica as a program transformation (Proposition \ref{prop:dial_to_ElensP}, Corollary \ref{cor:dial=functor}).

We conclude with some interesting questions in Section \ref{sec:conclusions_futurework}.


\begin{note}[Related works]

In Section \ref{sec:diff_lenses} we show that lens categories express reverse differentiation.
This is folklore and can be found e.g.\ in the preprints \cite[Example 3.7]{spivak2022generalizedlenscategoriesfunctors}, or \cite[Page 6]{Spivak20}.
However, that is only in the context of polynomial functors, with no mention of Dialectica nor (reverse) differential/tangent categories, while we use it (in Section \ref{sec:dial_is_lens}) precisely in relation with those notions in order to generalize some constructions sketched in \cite{KerjeanPedrot24}.
A similar fibrations/pullbacks based point of view on Dialectica, similar to the first part of our Section \ref{sec:dial_is_lens}, has just very recently been explored in the preprint 
    \cite{fibrationalPovLensesDialectica24}.
    However, the direction taken there is orthogonal to ours: we both look at Dialectica categories as categories of lenses, but while our attention is towards the links with reverse differentiation and program transformations (for which it turns out that trivial subobjetcts are enough), 
    they are interested in lifting to lenses the Dialectica transformation of formulas (and deal with non-trivial subobjects). 


Last, but not least, a special mention has to be made for the topics around the \emph{CHAD} construction in Automatic Differentiation (AD). We thank Fernando Lucatelli Nunes for recently pointing this important part of the literature to us.
Indeed, the very kind of constructions that we show here to be behind the work in \cite{KerjeanPedrot24}, are used in the \emph{CHAD} tradition in order to deal with AD for high-order programs.
\emph{A posteriori}, this is clearly no surprise.
This connection between that kind of AD and Dialectica has not been explicitly mentioned in \cite{KerjeanPedrot24} nor in e.g.\ \cite{chad21,chad22}, but an important link with Dialectica has very recently appeared in the preprint \cite{nunes_vakar_grot_dial2024}, in a general abstract categorical setting.
This is certainly relevant for many of the questions that we raise at the end.
In any case, as we explained before, our focus is not on AD and, by focussing on the formulation of Dialectica as in \cite{KerjeanPedrot24}, we try instead to give a handy, intuitive, and concrete discussion about the relations between Dialectica and differentiation.

In conclusion, all the above mentioned topics seem to converge to the same kind of mathematics, and if the expert in each of the topics will probably not be surprised while reading this paper, our purpose is, nevertheless, not to reinvent already existing topics; 
it is to relate them in order to understand and clarify, in a natural and geometrically intuitive way, what is the reason why Dialectica appears related to differentiation: to the best of our knowledge such a point of view 
has not been explicited before.
Therefore, we hope that this ``excursion'' will make this deep topic clearer and less mysterious, whether the reader comes from the world of Dialectica, of category theory/AD, or of $\lambda$-calculus.
\end{note}

\begin{note}
We write $f;g$ for the composition of $f:A\to B$ and $g:B\to C$ in a category.
We write $1_A$ for the identity arrow on $A$, and we drop $A$ when clear from the context.
If $\times$ denotes a product, we write $\pi_i^{A_1,A_2}:A_1\times A_2\to A_i$ for the projections (and we drop the ``$A_1,A_2$'' if clear).
If a category is closed, we denote $\curry/\uncurry$ its curry/uncurry operators.
If the products are symmetric (as it will always be), we keep implicit the obvious isos, so $\curry:\cC(A\times B,C)\to\cC(A,C^B)$ and $\curry:\cC(B\times A,C)\to\cC(A,C^B)$, 
and the same for uncurry.
Due to the space limitations, we take for granted the basic notions in all the topics that we cover, and we often point to the literature from which we take the mathematics that we need.
\end{note}

\section{Dialectica and reverse differentiation}\label{sec:log_rel}


\paragraph*{P\'edrot's style Dialectica}

The calculus is given in \cite[Fig.\ 1 and Def.\ 4.1]{KerjeanPedrot24}, but for the sake of clarity we recall here its main features.
The first layer of syntax is that of a simply typed $\lambda$-calculus with pairs (notation:$\langle M, N\rangle$ for pairs and $M^i$ for projections) and product types (notation: $A\times B$), quotiented under the usual $\beta\eta$-equality.
On top of that, we have a new monadic type constructor $\pmon \_$, together with its return and bind term constructors (notation: $[M]$ and $M\bind N$), and a commutative monoid structure on it (notation: $0$, $M+N$) which is compatible with the monad structure. 
All these equations (together with $\beta\eta$) constitute its equational semantics and are denoted by $=$.
Finally, since we are considering Dialectica as a transformation of proofs, not just of formulas, the transformation is now given by two maps $W,C$ from simple types to $\Pdial$-types and two maps $\wdial{(\_)}, \cdial{(\_)}{x}$ (for $x$ any variable) from $\lambda$-terms to $\Pdial$-terms, inductively defined in Figure \ref{fig:WC} and Figure \ref{fig:dial}.
In \cite{Pedrot14} one finds the soundness results mentioned in the introduction, as well as the computational interpretation of Dialectica, together with the proof of the fact that such transformation only depends on the equational semantics classes of $\stlc$.

\begin{figure}[t]
\begin{subfigure}{0.99\textwidth}
   \centering\boxed{
    \begin{tabular}{cc|ccc|cc}
          & & & $\alpha$ (ground types) & & & $E\to F$ \\
          \hline
        W & & & $\alpha_W$             & & & $(W(E)\to W(F))\times (W(E)\times C(F)\to \pmon{C(E)})$ \\
        C & & & $\alpha_C$             & & & $W(E)\times C(F)$
    \end{tabular}}
   \caption{Witnesses and Counters of a simple type. $\alpha_W$ and $\alpha_C$ are fixed ground types of $\Pdial$ associated with $\alpha$. In the second component of the witness of an arrow, we take a slightly different version, but equivalent, from the original, which would curry our type as $C(F)\to W(E)\to \pmon{C(E)}$. This is just because we want to highlight the intuition of the reverse differentials: with dependent types in mind, $W(E)\times C(F)=C(E\to F)$ plays the role of $\sum_{e:E} T_{fe}^*F=f^*T^*F$. \phantom{, so we could not give $v:C(F)$ \emph{before} $e:W(E)$ (what would $v$ be cotangent at?).}
   }
   \label{fig:WC} 
\end{subfigure}
\begin{subfigure}{0.99\textwidth}
   \centering\boxed{
    \begin{tabular}{c|cc|cc|cc}
                        & $x$  & & $\lambda x.M$ & & $PQ$ \\
        \hline
        $\wdial{(\_)}$  & $x$  & & $\langle\,\lambda x.\wdial M\,,\,\lambda\pi.(\lambda x.\cdial M x)\pi^1\pi^2\,\rangle$ & & ${\wdial P}^1\wdial Q$ \\
        $\cdial{(\_)}{y}$  &  $\begin{cases}
            \lambda\pi.[\pi], & x=y \\
            \lambda\pi.0, & x\neq y
        \end{cases}$ &  & $\lambda\pi.(\lambda x.\cdial M y)\pi^1\pi^2$ & & $\lambda\pi.( \cdial P y\langle\wdial Q,\pi\rangle+{\wdial P}^2\langle \wdial Q,\pi\rangle \bind \cdial Q y)$
    \end{tabular}}
   \caption{Untyped Dialectica transformation.
    Remark that we take a slightly different version, but equivalent, than the original, in order to fit with the modification mentioned above.
    Notice that 
    $\wdial{(\lambda x.M)}$ 
    is reminiscent of the pair ``$(f,f_*)$''.}
   \label{fig:dial}
\end{subfigure}
\caption{The Dialectica program transformation for the $\stlc$ in $\Pdial$.}
\label{fig:Pdial}
\end{figure}

\paragraph*{Relating Dialectica and Reverse Differentials}

The constructions of this section take place in a generic model $\cC$ of classical Differential Linear Logic, which we fix below following \cite{KerjeanPedrot24}.
One example is found in \cite[Theorem 7.16]{mod_LL_schwartz19}.

\begin{note}
We fix a left-additive category $\cC$ enriched over commutative monoids (we use $0,+$ for the operations on the homsets) endowed with:
    a symmetric monoidal closed structure $(\otimes, 1)$, whose exponential objects we denote $\ilhom{A}{B}$ and evaluation arrows $\ev~:~\ilhom{A}{B}\otimes A \lto B$; 
    finite biproducts $(\&, \top)$ (with projections $\pi_i$); 
    a strong monoidal comonad $(!:\cC\to\cC,$
    $d~:~!~\to~\mathrm{id},\, p:!\to !!)$ (resp.\ called \emph{bang}, \emph{dereliction}, \emph{digging});
    natural transformations $\ctr:!\to !\otimes !$ (\emph{contraction}) and $\wkn:!\to !\top$ (\emph{weakening}) making $!$ a storage modality; 
    isomorphisms $!A\otimes !B \simeq !(A\& B)$ and $1\simeq !\top$ making $\cC$ Seely; 
    a natural transformation $\coder:\id \to !$ (\emph{codereliction})  making $\cC$ differential storage; 
    an involutive functor $\dual{(\_)}:\cC^{\mathrm{op}}\to \cC$ making $\cC$ $\star$-autonomous with a natural bijection $\chi:\cC(D\otimes E,F) \simeq \cC(D,\ilhom{\dual F}{\dual E})$.
This means that a series of equations are required, for which we refer to standard references (e.g.\ \cite{mellies_cat_mod_LL}).
We systematically use the notation $A\lto B$ for arrows in $\cC$.
\end{note}

In such setting we can make sense of differential notions, as we describe in the note below.
We highlight in our presentation the inspiration from Differential Geometry.

\begin{note}
It is well-known that with the above data one can also define natural transformations $\coctr: !\otimes ! \to !$ (\emph{cocontraction}), $\cowkn: !\top \to !$ (\emph{coweakening}) and
$\partial:
\id\otimes ! \overset{\coder\otimes 1}{\lto} !\otimes ! \overset{\coctr}{\lto} !$ (\emph{deriving transformation}), and set the \emph{differential of $f:!A\lto B$ in $\cC$} be $\partial f:=A\otimes!A\overset{\partial}{\lto} !A \overset{f}{\lto} B$.
It is well-known that the coKleisli $\cC_!$ (same objects as $\cC$ and $\cC_!(A,B):=\cC(!A,B)$, representing non-linear arrows) is a Cartesian closed differential category.
We systematically use the notation $A\to B$ for arrows in $\cC_!$.
Its products are $A\times B:= A\& B$ with projections $!\pi_i;d$, its exponential objects $\ihom A B := \ilhom{!A}{B}$ and the \emph{differential of $f\in \cC_!(A,B)$ in $\cC_!$} is $Df:=d\otimes 1 ; \partial f \in \cC_!(A\times A,B)$.
For $f:A\to B$ we have its \emph{promotion} $\prom f:=!A\overset{p}{\lto} !! A \overset{!f}{\lto} !B$. 
For (finitely many) elements $a_i:\top\to A_i$ of $A_i$ (notation: $a_i:A_i$) and $f:\prod_i A_i\to B$, we define the element $f(\vec a):B$ as $1 \overset{\simeq}{\lto} \bigotimes_i 1 \overset{\bigotimes_i \prom a_{\!\!i}}{\lto} \bigotimes_i !A_i \overset{f}{\lto} B$.
For $f:A\to B$ we call its \emph{pushforward} the arrow $\pshfwd f:A\overset{\curry \partial f}{\to} \ilhom{A}{B}$ and $\pshfwd f a:A\overset{\uncurry(\pshfwd f(a))}{\lto} B$ its pushforward at $a:A$.
Finally, for $f: A\lto \ilhom{E}{F}$ and $e:E$, we let $\restr{f}{e} := A \overset{\simeq}{\lto} A\otimes 1 \overset{1\otimes e}{\lto} A\otimes E \overset{\uncurry f}{\lto} F$.
\end{note}

\begin{figure}[t]
  \centering\boxed{
    \begin{tabular}{c}
        $\begin{array}{c}
             \textit{For } \vdashPdial M:(W(E)\to W(F)) \times (W(E)\times C(F)\to \pmon{C(E)}) \textit{ and } f:[E\to F]\textit{,
        we set} \\ \\
        M \sim_{E\to F} f \textit{ iff 
        for all } H\sim_E e\textit{, we have } \\ \\
        M^1 H \sim_F \restr{f}{e} : F
        \quad\textit{ and } \quad
            \lambda \pi. M^2\langle H,\pi\rangle \bowtie_F^E
            (\quad \uncurry f:E\to F\quad,\quad
            \dual{(\pshfwd{(\uncurry f)}{e})} : \dual F\lto \dual E\quad ) \\ \\
        \end{array}$
         \\ \hline \\
        $\begin{array}{c}
        \textit{For } \vdashPdial M:W(E)\times C(F)\to \pmon{W(A)}, f:A\to [E\to F], g:[E\to F]\lto A\textit{, we set } \\ \\
        M \bowtie_{E\to F}^A (f,g) \textit{ iff
        for all } H\sim_E e\textit{, we have } \\ \\
        \lambda\pi.M\langle H,\pi\rangle \bowtie_{F}^A
            (\quad \restr{f}{e}:A\to F \quad,\quad
            {\restr{\dual g}{e}}^{\!\!\perp} : \dual F\lto \dual A \quad ) \\ \\
        \end{array}$
    \end{tabular}}
    \caption{Definition \ref{def:log_rels} of relations $\sim$ and $\bowtie$, inductive case $B=E\to F$ (this is the only case).}
    \label{fig:log_rel}
\end{figure}

We fix now an interpretation of ground simple types in $\cC$ extended in the canonical way to all simple types (we still write $A$ for the interpretation of the simple type $A$ in $\cC$).
Taking inspiration from \cite{KerjeanPedrot24}, we define two logical relations $\sim$ and $\bowtie$ relating a closed Dialectica-transformed program of $\Pdial$ (i.e.\ the image of a proof under Dialectica) with arrows, of the suited type, in the ambient category.

\begin{definition}\label{def:log_rels}
    Given, for any ground type $\alpha$ and simple type $A$, two relations $\sim_{\alpha}\subseteq \set{\vdashPdial M:\alpha}\times \cC_!(\top,\alpha)$ and $\bowtie_\alpha^A \subseteq \set{\vdashPdial M:\alpha \to \mathcal{M}_A}\times \cC_!(A,\alpha)\times \cC(\dual\alpha,\dual A)$, we lift them at all simple types $B$ in order to get relations
    \[\begin{array}{lcccc}
        \sim_{B} & \subseteq & \set{\vdashPdial M:W(B)} & \times & \cC_!(\top,B) \\
        \bowtie_B^A & \subseteq & \set{\vdashPdial M:C(B) \to \pmon{C(A)}} & \times & \cC_!(A,B)\times \cC(\dual B,\dual A) 
    \end{array}\]
    defined by mutual induction on $B$ as in Figure \ref{fig:log_rel}.
\end{definition}

The expert differential $\lambda$-calculist could relate the arrows in Definition~\ref{def:log_rels} and the terms of \cite[Def.\ 4.8]{KerjeanPedrot24}, or wait Proposition~\ref{prop:Pm_to_me} where we show in which sense the two constructions are equivalent.
But (one of) the points of our definition is to make these constructions explicit using the familiar operations of pushforwards and dual maps.
Finally, remembering cotangent spaces and their dependently typed nature, the $g$ in
$M \bowtie_B^A (f,g)$ should be intuitively understood as $g:T_{f(a)}^*B \lto T_a^*A$, for an $a\in A$.



\begin{lemma}\label{lm:lift}
    Suppose $\sim_{\alpha}, \bowtie_\alpha^A$ are 
    closed w.r.t.\ the rules of Figure \ref{fig:rules}.
    Then the same holds for $\sim_B$ and $\bowtie_B^A$ for all simple types $B$.
\end{lemma}
\begin{proof}
    Each rule is proved separately by induction on $B$, except rules $(\ev)$ and $(d)$ which are proved by mutual induction on $B$.
    The lift of the compatibility with equational equivalence is immediate.
    The others are all straightforward using the equational semantics of $\Pdial$ \cite[Def.\ 4.1]{KerjeanPedrot24} and equations which hold in $\cC$:
    $(0)$ uses $\dual 0=0$.
    $(+)$ uses $\dual{(f+g)}=\dual f+\dual g$, $\uncurry{(f+g)}=\uncurry f + \uncurry g$ and $f;(g+h)=f;g+f;h$.
    $(\bind)$ uses $\restr{(\prom h;f)}{e}=\prom h;\restr{f}{e}$ and $\dual{\restr{\dual{(s;g)}}{e}}=\dual{\restr{\dual{s}}{e}};g$.
    $(d)$ uses the inductive hypothesis on $(\mathrm{eval})$ and $\mathrm{eval}_{e}=\restr{d}{e}$.
    $(\mathrm{eval})$ uses the fact that $\restr{\mathrm{eval}_{\vec a}}{e}=\mathrm{eval}_{\vec a,e}$.
\end{proof}

\begin{figure}[t]
    \centering
    \boxed{
    \begin{tabular}{lcr}
    $\dfrac{M\sim_\alpha f}{N\sim_\alpha f}(\textit{if }M=N)$     & \multicolumn{2}{l}{\quad $\dfrac{M\bowtie_\alpha^A (f,g)}{N\bowtie_\alpha^A (f,g)}(\textit{if }M=N)$    \quad    $\dfrac{ G \bowtie_D^A (h,g) \qquad M \bowtie_\alpha^D (f,s) }{\lambda \pi.(M\pi \bind G) \bowtie_\alpha^A (\prom h;f,s;g)}(\bind)$}     \\
    \\
    $\dfrac{}{\lambda \pi.0 \bowtie_\alpha^A \binom{f:A\to \alpha}{0:\dual\alpha \lto \dual A}}(0)$     & $\dfrac{M_1 \bowtie_\alpha^A \binom{f:A\to \alpha}{ g_1:\dual\alpha \lto \dual A} \qquad M_2 \bowtie_\alpha^A
    \binom{f:A\to \alpha}{ g_2:\dual\alpha \lto \dual A}}{\lambda \pi.(M_1\pi\plus  M_2\pi) \bowtie_\alpha^A \binom{f:A\to \alpha}{ g_1+g_2:\dual\alpha \lto \dual A}}(+)$     & $\dfrac{}{\lambda\pi.[\pi] \bowtie_\alpha^\alpha (d,1)}(d)$        \\
    \\
    \multicolumn{3}{c}{$\dfrac{M_1 \sim_{A_1} a_1:A_1 \quad \overset{(n\geq 1)}{\dots} \quad M_n \sim_{A_n} a_n:A_n}{\lambda \pi. [\langle M_1,\dots,M_n,\pi \rangle]
    \bowtie_\alpha^A
    \binom{\mathrm{eval}_{\vec a}:[A_1\to [A_2\to \dots \to [A_n\to \alpha]]]\to \alpha}{\dual{(\pshfwd{(\mathrm{eval}_{\vec a})}0)}:
    \dual\alpha \lto [[[
    \dual{\alpha} \lto 
    \dual A_n] \lto
    \dual A_{n-1}] \lto 
    \dots \lto \dual A_1
    ]}}(\mathrm{eval})$} \\ \\
    \end{tabular}
    }
    \caption{
    In $(\mathrm{eval})$, we let $\mathrm{eval}_{\vec a}:\ihom{A_1}{\ihom{A_2}{\cdots\to\ihom{A_n}{B}}}\to B$ be the composition 
    $[1,n]\overset{d}{\to}[1,n]\overset{h_1}{\to}[2,n]
    \to \cdots
    \to [n,n]\overset{h_n}{\to} B$, where 
    $[i,n]:=\ihom{A_i}{\ihom{A_{i+1}}{\cdots\to\ihom{A_n}{B}}}$ and $h_j:=\restr{1_{A_j}}{\prom a_{\!\!j}}$.}
    \label{fig:rules}
\end{figure}

For $f:\prod_i A_i\to B$ and $a_i:A_i$ for $i=1,\dots,j-1,j+1,\dots,n$, we let $f_{\vec a}^j:=!A_j\simeq \bigotimes_{1}^{j-1} 1\otimes \,!A_j\otimes \bigotimes_{1}^{n-j}  1 \lto \bigotimes_i !A_i \overset{f}{\lto} B$, where the unnamed arrow is $\bigotimes_i \prom a_{\!\!i}\otimes 1\otimes \bigotimes_i \prom a_{\!\!i}$.
This corresponds to $f$ where we fixed all inputs at the $a_i$'s but the $j$-th one.

The following statement expresses the chain rule in its pushforward form\footnote{The usual undergraduate chain rule is obtained when $p$ does not depend on $A_j$.}.
The differential $\lambda$-calculist will notice that it precisely corresponds to the \emph{linear substitution} of an application.

\begin{lemma}\label{lm:chain_rule}
    Let $p:\prod_i A_i\to \ihom E F$ and $q:\prod_i A_i \to E$.
    Let us momentarily write $q\,;_!p:=
    \prod_i A_i 
    \overset{c}{\lto} \bigotimes_i !A_i \otimes \bigotimes_i !A_i \overset{p\otimes\prom q}{\lto} \ihom E F \otimes !E \overset{\ev}{\lto} F$, which is the ``evaluation'' of $p$ at $q$ in $\cC_!$.
    Fix $j$ and $a_i:A_i$ ($1\leq i\neq j\leq n$).
    In $\cC$ we have:
    \[{\pshfwd{((q\,;_!p)_{\vec a^j}^j)}a_j}
    =
    {\restr{\pshfwd{(p_{\vec a^j}^j)}{a_j}}{q(\vec a)}} 
    +\,
    {(\pshfwd{(q_{\vec a^j}^j)}{a_j}}\,;\,{(\pshfwd{\uncurry{(p(\vec a))})}{q(\vec a)})}.\]
\end{lemma}

From now on, we fix an interpretation $\llbracket\cdot\rrbracket:ST\lambda C\to\cC_!$ and ground relations $\sim,\bowtie$ satisfying the hypotheses of Lemma \ref{lm:lift}.

\begin{theorem}\label{th:main}
    Let $f:=\llbracket x:A_1,\dots,x:A_n\vdash_{\stlc} M:B \rrbracket : \prod_{i=1}^n A_i \to B$.
    
    For all $\vdash_{\stlc} N_i:A_i$
    and $a_i:A_i$ s.t.\ $N_i \sim_{A_i} a_i$ ( $i=1,\dots,n$), setting $\vec a:=a_1,\dots,a_n$ and $\vec a^j:=a_1,\dots,a_{j-1},a_{j+1},\dots,a_n$, we have:
    \begin{enumerate}
        \item $M^\bullet\set{\vec N/\vec x} \sim_B f(\vec a)$.
        \item If $1\leq j\leq n\neq 0$, the following rule is admissible for all simple type $Y$:
    \[
    \dfrac{G \bowtie_{A_j}^{Y} ( \, h:!Y\lto A_j\, , \, g: \dual A_j\lto \dual Y \, )}
    {\lambda \pi. ((M_{x_j}\set{\vec N/\vec x})\pi \bind G) \bowtie_{B}^{Y}
    (\,\prom h;f_{\vec a^j}^j:!Y\lto B\,\,,\,\,
    \dual{(\pshfwd{(f_{\vec a^j}^j)}a_j)}; g : \dual B \lto \dual Y
    \,)}
    \]
    \end{enumerate}
\end{theorem}
\begin{proof}
    Induction on $M$. Call (IH1), (IH2) the inductive hypotheses for claim 1,2.
    \begin{description}
        \item[Case $M=x_i$.]
    Then $f=\pi_i$.
    
    1). Our goal becomes $N_i\sim_{A_i} a_i$ which is in our hypotheses.
    
    2). 
    If $j=i$, we have $f_{\vec a^j}^j=d_{A_j}$ and one can show that $\pshfwd{(f_{\vec a^j}^j)}a_j=\pshfwd{d}{a_j}=1$.
    Our goal then becomes $G\bowtie_{A_j}^Y (h,g)$, which is precisely the premise of our rule.
    If $j\neq i$, We have $f_{\vec a^j}^j=\prom w_{\!\!A_j};a_i$ and one can show that $\pshfwd{(f_{\vec a^j}^j)}a_j=\pshfwd{(\prom w_{\!\!A_j};a_i)}{a_j}=0$. 
    Our goal then becomes $\lambda \pi.0 \bowtie_{A_i}^Y (\prom h;f_{\vec a^j}^j\,,\,0)$, which is given by $(0)$.

    \item[Case $M=\lambda y.Q$, $B=E\to F$].
    Then there is $\uncurry f=\llbracket \vec x:\vec A,y:E\vdash_{\stlc} Q:F \rrbracket$.
    
    1).
    We have to show that, given $H\sim_E e$, we have both $\wdial Q\set{\vec N/\vec x,H/y} \sim_F \restr{f(\vec a)}{e}$ and $\cdial{Q}{y}\set{\vec N/\vec x,H/y} \bowtie_F^E (\,
    \uncurry{(f(\vec a))}\,\,,\,\,
    \dual{(\pshfwd{(\uncurry{(f(\vec a))})}{e})}\,)$.
    The former is given by (IH1), since $\restr{f(\vec a)}{e}=(\uncurry f)(\vec a, e)$.
    For the latter we have $\cdial{Q}{y}\set{\vec N/\vec x,H/y} = \lambda\rho.(\cdial{Q}{y}\set{\vec N/\vec x,H/y} \rho \bind \lambda\eta.[\eta])$ so, using rule $(d)$, this is precisely given by (IH2), since $\prom d_{\!\!E};(\uncurry f)^{n+1}_{\vec a}=\uncurry{(f(\vec a))}$.
    
    2).
    Given $G \bowtie_{A_j}^{Y} ( \, h:!Y\lto A_j\, , \, g: \dual A_j\lto \dual Y \, )$ and $H\sim_E e$, putting $P:=M_{x_j}\set{\vec N/\vec x})\pi \bind G$, our goal is: $\widetilde P \bowtie_F^Y 
    (\,\restr{(\prom h;f_{\vec a^j}^j)}{e}\,,\,
    \dual{\restr{\dual{(\dual{(\pshfwd{(f_{\vec a^j}^j)}{a_j})};g)}}{e}}\,)
    $, where we put $\widetilde P:=\lambda\rho.(\lambda \pi. P)\langle H,\rho \rangle$.
    Since $P=(\lambda y.\cdial{Q}{x_j}\set{\vec N/\vec x})\pi^1\pi^2\bind G$, we have $\widetilde P=\lambda\rho.(\lambda \pi. P)\langle H,\rho \rangle=\lambda \rho.(\cdial{Q}{x_j}\set{\vec N/\vec x,H/y}\rho\bind G)$.
    Now, by (IH2) on $Q$ with $(G,h,g)$, we obtain $\widetilde P\bowtie_F^Y (\, \prom h;(\uncurry f)^j_{\vec a^j,e} \,,\, \dual{(\pshfwd{((\uncurry f)^j_{\vec a^j,e})}{a_j})};g \,)$.
    To conclude, one can see that $\restr{(\prom h;f_{\vec a^j}^j)}{e}=\prom h;\restr{f_{\vec a^j}^j}{e}=\prom h;(\uncurry f)^j_{\vec a^j,e}$ and $\restr{(\pshfwd{(f_{\vec a^j}^j)}{a_j})}{e}=\pshfwd{((\uncurry f)^j_{\vec a^j,e})}{a_j}$ as well as $\dual{\restr{\dual{(\dual{(\pshfwd{(f_{\vec a^j}^j)}{a_j})};g)}}{e}}=
    \dual{\restr{(\pshfwd{(f_{\vec a^j}^j)}{a_j})}{e}};g
    $.
    
    \item[Case $M=PQ$.]
    Then $f=c
    ;(p\otimes \prom q);\ev$, where $p=\llbracket \vec x:\vec A\vdash_\stlc P:E\to B\rrbracket$ and $q=\llbracket \vec x:\vec A\vdash_\stlc Q:E \rrbracket$.
    
    1).
    Since one sees that $f(\vec a)=\restr{p(\vec a)}{q(\vec a)}$, our goal becomes showing that $(\wdial{P}\set{\vec N/\vec x})^1(\wdial Q\set{\vec N/\vec x})$ $\sim_B \restr{p(\vec a)}{q(\vec a)}$.
    But (IH1) on $P$ gives $(\wdial{P}\set{\vec N/\vec x})^1 H\sim_{B} \restr{p(\vec a)}{e}$ for all $H\sim_E e$, and (IH1) on $Q$ gives $\wdial{Q}\set{\vec N/\vec x}\sim_{E} q(\vec a)$, so we are done.
    
    2).
    Given $G \bowtie_{A_j}^{Y} ( \, h:!Y\lto A_j\, , \, g: \dual A_j\lto \dual Y \, )$, let $R:=\lambda \eta.((\cdial{Q}{x_j}\set{\vec N/\vec x}\eta)\bind G)$, 
    $\widetilde P_{\rho}:=(\cdial{P}{x_j}\set{\vec N/\vec x}\langle \wdial Q\set{\vec N/\vec x},\rho \rangle)\bind G$
    and $\widetilde Q_{\rho}:=({\wdial{P}}^2\set{\vec N/\vec x}\langle\wdial Q\set{\vec N/\vec x},\rho\rangle)\bind R$.
    Now our goal is:
    $\lambda\pi.((\lambda\rho.\widetilde P_{\rho})\pi+(\lambda\rho.\widetilde Q_{\rho})\pi) \bowtie_B^Y
    (\,\prom h;f_{\vec a^j}^j\,\,,\,\,
    \dual{(\pshfwd{(f_{\vec a^j}^j)}a_j)}; g
    \,)
    $.
    By $(+)$ and Lemma \ref{lm:chain_rule}, it is enough showing that we have both $\lambda \rho.\widetilde P_{\rho} \bowtie_B^Y (\prom h;f_{\vec a^j}^j\,,\, \restr{\pshfwd{(p_{\vec a^j}^j)}{a_j}}{q(\vec a)}^{\!\!\!\!\!\!\!\!\!\bot}\,\,\,\,;g)$ and also $\lambda\rho.\widetilde Q_{\rho}\bowtie_B^Y$ $(\prom h;f_{\vec a^j}^j,$ $
    \dual{(\pshfwd{(\uncurry{(p(\vec a))})}{q(\vec a)})}
    ;\dual{(\pshfwd{(q_{\vec a^j}^j)}{a_j})}
    ;g\,)$.
    For the former, IH1 on $Q$ entails $\wdial Q\set{\vec N/\vec x}\sim_E~q(\vec a)$.
    So one can see that IH2 on $P$ precisely gives $\lambda \rho.\widetilde P_{\rho}\bowtie_B^Y (\,
    \restr{(\prom h;p_{\vec a^j}^j)}{q(\vec a)} \,,\,
    \restr{\dual{(\dual{(\pshfwd{(p_{\vec a^j}^j)}a_j)}; g)}}{q(\vec a)}^{\!\!\!\!\!\!\!\!\!\bot}\,\,\,\,\,)$, and it is easy to see that we obtained the desired pair of arrows.
    For the latter, on the one hand we notice that, by IH2 on $Q$, we have $R\bowtie_{B}^{E}
    (\,\prom h;q_{\vec a^j}^j\,\,,\,\,
    \dual{(\pshfwd{(q_{\vec a^j}^j)}a_j)}; g
    \,)$.
    On the other hand, by IH1 on $P$, for all $H\sim_E e$ we have $\lambda\pi.{\wdial P}^2\set{\vec N/\vec x}\langle H,\pi\rangle \bowtie_B^E (\,\uncurry{(p(\vec a))}\,,\, \dual{(\pshfwd{(\uncurry{(p(\vec a))})}{e})}\,)$.
    Now, we already remarked above that $\wdial Q\set{\vec N/\vec x}\sim_E q(\vec a)$, thus putting $S:=\lambda\pi.{\wdial P}^2\set{\vec N/\vec x}\langle \wdial Q\set{\vec N/\vec x},\pi\rangle$, we have $S\bowtie_B^E (\,\uncurry{(p(\vec a))}\,,\, \dual{(\pshfwd{(\uncurry{(p(\vec a))})}{q(\vec a)})}\,)$.
    But by rule $(\bind)$ on $R$ and $S$, we obtain $\lambda\rho.(S\rho\bind R) \bowtie_B^Y (\,
    \prom{(\prom h;q_{\vec a^j}^j)};\uncurry{(p(\vec a))}
    \,,$
    $
    \dual{(\pshfwd{(\uncurry{(p(\vec a))})}{q(\vec a)})};\dual{(\pshfwd{(q_{\vec a^j}^j)}a_j)}; g
    \,)$.
    Now since $S\rho\bind R=\widetilde Q_{\rho}$, one concludes by checking that $\prom{(\prom h;q_{\vec a^j}^j)};\uncurry{(p(\vec a))}=\prom h;f_{\vec a^j}^j$.
    \end{description}
\end{proof}

Now using rule $(d)$ as premise of the rule in Theorem \ref{th:main}(2), we obtain:

\begin{corollary}\label{cor:dial_rev}
    Under the hypotheses of Theorem \ref{th:main}(2), we have \[M_{x_j}\set{\vec N/\vec x} \bowtie_{B}^A
    (\, f_{\vec a^j}^j:A_j\to B\,\,,\,\,\dual{(\pshfwd{(f_{\vec a^j}^j)}a_j)} : \dual B\lto \dual A_j \,).\]
\end{corollary}

\begin{remark}\label{rmk:dial_rev}
    For $x:A\vdash_\stlc M:B$, the results above say that $\wdial{(\lambda x.M)} \sim_{A\to B} \llbracket M \rrbracket$ and $(\lambda x.\cdial{M}{x})N\bowtie_B^A (\llbracket M \rrbracket,\dual{(\pshfwd{\llbracket M \rrbracket}{a})})$ for all $N\sim_A a$.
    Remembering the reverse differential $R\llbracket M \rrbracket:(a,w)\in\plb{\llbracket M \rrbracket}{\,\cotan B}\mapsto (a,\dual{(\pshfwd{\llbracket M \rrbracket}{a})}w)\in\cotan A$ of $\llbracket M \rrbracket$, we can read it by saying that $\lambda x.\cdial{M}{x}$ ``represents'' $R\llbracket M \rrbracket$.
    Precisely spelling this out a dependently typed framework would be very interesting.
\end{remark}

The previous two results and the remark above express the Dialectica as a differentiable program transformation in a categorical way, hopefully clarifying even more the content of \cite[Theorem 4.10]{KerjeanPedrot24} (compare also \cite[Fig.\ 6]{KerjeanPedrot24} with our Definition \ref{def:log_rels}).

\begin{remark}
    $\sim$ can be thought of as a ``proof relevant'' realisability relation: not only do we realise a formula $B$ with $\Pdial$-terms (the $M$'s such that $M\sim_B f$, \emph{for some} $f:B$), but we also cluster such realisers into classes whose terms realise a certain element $f:B$ (the statement that $M\sim_B f$).
    Theorem \ref{th:main} becomes then the usual adequacy Theorem for realisability. 
\end{remark}

\begin{figure}[t]
    \centering
    \boxed{
    \begin{tabular}{c}
     $\dfrac{M\sim_B \llbracket S\rrbracket}{M\overset{\partial\lambda}{\sim}_B S}(1)$ \quad $\dfrac{M\bowtie_B^A (f,\dual{\llbracket S\rrbracket})}{M\overset{\,\,\,\partial\lambda_A}{\bowtie}_{\!\!\!\!B} S}(2)$ \quad
     $\dfrac{M\overset{\partial\lambda}{\sim}_B S}{M\sim_B \llbracket S\rrbracket}(3)$ \quad $\dfrac{M\overset{\,\,\,\partial\lambda_A}{\bowtie}_{\!\!\!\!B} S \quad \llbracket S\rrbracket = \pshfwd{f}{a}}{M\bowtie_B^A (f,\dual{\llbracket S\rrbracket})}(4)$ \\ \\
    \end{tabular}}
    \caption{From  relations in \cite[Def.\ 4.8]{KerjeanPedrot24} (denoted $\overset{\partial\lambda}{\sim}$, $\overset{\partial\lambda}{\bowtie}$) to ours in Definition \ref{def:log_rels}, and vice versa.
    $S$ is a $\stdlc$-term.
    The careful reader would notice that, rigorously speaking, one needs to slightly modify the term $M$ when passing from the formulation in \protect\cite{KerjeanPedrot24} to ours, because of the modifications mentioned at Figure \ref{fig:dial}. We leave it implicit since it is easy to recover by following 
    Figure \ref{fig:WC}.}
    \label{fig:PM_to_me}
\end{figure}

As previously mentioned, the following result explains the relation between \cite[Def.\ 4.8]{KerjeanPedrot24} and our Definition \ref{def:log_rels}: the latter appears slightly more general than the former, due to the supplementary hypothesis on $\cC$ needed to have an equivalence.
The proof is by straightforward mutual induction on $B$ (using the relation in \cite{KerjeanPedrot24}).

\begin{proposition}\label{prop:Pm_to_me}
    Let $\llbracket\_\rrbracket$ be an interpretation $\stdlc\to\cC_!$.
    Suppose that the rules of Figure \ref{fig:PM_to_me} hold when $B$ is a ground type.
    If $\llbracket\_\rrbracket$ is fully complete (i.e.\ surjective on all homsets), then the rules lift to all simple type $B$.
    Moreover, \cite[Theorem 4.10]{KerjeanPedrot24} follows from our Theorem \ref{th:main} using $(1)$, $(2)$, and our Theorem \ref{th:main} follows from theirs using $(3)$, $(4)$.
\end{proposition}


\paragraph*{Is Dialectica \emph{really} Differentiation?}

The following example is due to personal discussions with Thomas Powell, which we thank:
Imagine having a theory for arithmetic on real numbers, which allows for extensional reasoning about equality, via a substitution rule allowing to derive $A\set{t/x}\vdash A\set{s/x}$ from $\vdash t=s$, where $A$ is any formula and $t,s$ are terms (of type $\mathrm{real}$).
Suppose also that we have the axiom $x={(x^{3})}^{1/3}$ in our theory.
Then we have the derivation $M$ in Figure~\ref{fig:tom}.

Applying the Dialectica transformation to $M$, e.g.\ in the style of the recent preprint~\cite{barba_powell_25}, would extract terms $M^\bullet=\langle\lambda x^{\mathsf{real}}.x^3 \, , \, \lambda x^{\mathsf{real}} \pi^{\mathsf{real}}.\pi\rangle:W(\exists x\forall y.\, x=y \to \exists u\forall y.\, u^{1/3}=~y)$ witnessing the derivation $M$.
Now one would like to apply some sort of Remark \ref{rmk:dial_rev} in order to say that those represent a function and its reverse differential.
However, it is clearly not the case that the function $g:\mathbb R \times \mathbb R^{\bot}\to \mathbb R^{\bot}$, $g(x,\pi)(v)=\pi(v)$ is the reverse differential of $f:\mathbb R\to\mathbb R$, $f(x)=x^3$, the correct one being $Rf:\mathbb R\times \mathbb R^{\bot}\to \mathbb R^{\bot}$, $Rf(x,\pi)(v)=(\pshfwd{f}{x})^\bot(\pi)(v)=\pi((\pshfwd{f}{x})(v))=\pi(3x^2v)$.

Is this in contradiction with \cite{KerjeanPedrot24}, where it is claimed that Dialectica computes the differentials even of high order functions?
The answer is no.
In fact, 
as we are hopefully going to clarify in the following, what is shown in \cite{KerjeanPedrot24} is that Dialectica composes as reverse differentiation, and therefore if we start with functions and their reverse differentials in the language, Dialectica preserves the second being the reverse differential of the first even for compound high order functions. 
This is also why it makes sense to only consider the simply typed $\lambda$-calculus: it is just what is needed in order to treat composition (and the definition) of high order functions. 
The previous example shows instead something different: if we consider Dialectica on a richer logic, with quantifiers and substitutions, then the extracted witness functions are not necessarily in a function/differential relation.

\begin{remark}
    Observe that, even in the case above where Dialectica does not compute the differential, the approach taken in \cite{barba_powell_25} shows that one can still understand its realizers via a generalized backpropagation relation. This is still related to the compositional aspect of it, i.e.\ it is shown that it can still be understood as a sort of reverse chain rule, and this remains coherent with \cite{KerjeanPedrot24} and the example above.
\end{remark}

In the following, we are going to clarify this by taking a categorical approach: while the above example extracts Dialectica realizers and looks at them as actual functions, \cite{KerjeanPedrot24} and our Remark \ref{rmk:dial_rev} only look at the structural (read: compositional) properties of the Dialectica extraction process, i.e.\ the transformation for the $\stlc$. The former can be understood in terms of De Paiva's Dialectica Categories (which abstracts the orthogonality relation as non-trivial subobjects); the latter, instead, as we are going to see, in terms of 
lens categories.

\begin{figure}[t]
\centering
\begin{prooftree}
\AxiomC{}
\RightLabel{\scriptsize ax}
\UnaryInfC{$\vdash x={(x^{3})}^{1/3}$}
\RightLabel{\scriptsize subst}
\UnaryInfC{$x=y \vdash (x^3)^{1/3}=y$}
\RightLabel{\scriptsize $\forall L$}
\UnaryInfC{$\forall y.\, x=y \vdash (x^3)^{1/3}=y$}
\RightLabel{\scriptsize $\forall R$}
\UnaryInfC{$\forall y.\, x=y \vdash \forall y.\, (x^3)^{1/3}=y$}
\RightLabel{\scriptsize $\exists R$}
\UnaryInfC{$\forall y.\, x=y \vdash \exists u\forall y.\, u^{1/3}=y$}
\RightLabel{\scriptsize $\exists L$}
\UnaryInfC{$\exists x\forall y.\, x=y \vdash \exists u\forall y.\, u^{1/3}=y$}
\end{prooftree}
\caption{Dialectica extracted realizers of this proof are not in a ``function/differential'' relation.
}
\label{fig:tom}
\end{figure}

\section{Reverse Differentiation and Lenses}\label{sec:diff_lenses}

As we have seen in the previous section, Dialectica can be read as a program transformation which mimics the construction of the reverse differential of a morphism.
In this section we see the general framework in which reverse differentiation takes place, namely that of (dependent) lenses.
We fix for all this section a category $\mathcal{L}$ with pullbacks.
The canonical example that we have in mind is the category $\mathbf{SMan}$ of smooth manifolds and smooth maps.

\begin{note}
We denote with $A \overset{\,\,\, \plb{f}{p}}{\leftarrow} \plb f \beta \overset{\overline{f}}{\to} \beta$ the pullback of a diagram $A \overset{f}{\to} B \overset{p}{\leftarrow} \beta$ (we use standard notation 
similar to e.g.\ \cite{hatcher_k_theory,diff_geo_book} for the fibre bundle $\plb{f}{p}:\plb{f}{\beta}\to A$).
For example,
    the span $A \overset{\,\,\, p}{\leftarrow} \alpha \overset{1}{\to} \alpha$ is the pullback of the diagram $A \overset{1}{\to} A \overset{p}{\leftarrow} \alpha$, and with our notations we write $\plb{(1_A)}{\alpha}=\alpha$, $\overline{1_A}=1_\alpha$ and $\plb{(1_A)}{p}=p$.
    As another example, if $\times$ is a Cartesian product, one can take the pullback of a projection along any arrow: the pullback of $A \overset{f}{\to} B \overset{\pi_1}{\leftarrow} B\times Y$ is given by $\plb{f}{(B\times Y)}=A\times Y$, $\plb{f}{\pi_1^{B,Y}}=\pi_1^{A,Y}$ and $\overline{f}=f\times 1_Y$. 
\end{note}


For the purpose of this paper, we take the following:

\begin{definition}\label{def:lenses}
The category $\mathrm{Lens}(\mathcal{L})$ of lenses over $\mathcal{L}$ is defined as follows:
objects: arrows in $\mathcal{L}$, which we think of as fibre bundles and we write $p:\binom{\alpha}{A}$ instead of $p:\alpha\to A$;
arrows from $p:\binom{\alpha}{A}$ to $q:\binom{\beta}{B}$ are the data of both
    a $f:A\to B$ in $\mathcal{L}$
    and a span $\alpha \overset{\,\,\,F}{\leftarrow} \plb{f}{\beta} \overset{\overline{f}}{\to} \beta$ in $\mathcal{L}$, such that in $\mathcal{L}$ the pullback square of Figure \ref{fig:diag_lenses} holds, and the left triangle 
    commutes.
The identity on $p:\binom{\alpha}{A}$ is given by $1_A$ and the identity span $\alpha \overset{\,\,\,1}{\leftarrow} \alpha \overset{1}{\to} \alpha$.
Composition is given by pairwise composition in $\mathcal{L}$ and composition in the category 
of spans on $\mathcal{L}$.
One can check that these 
satisfy the conditions for being arrows in $\mathrm{Lens}(\mathcal{L})$.
\end{definition}

\begin{figure}[t]
    \centering
    \boxed{
    \begin{tikzcd}[ampersand replacement=\&,cramped]
	\alpha \& {\plb{f}{\beta}} \& \beta \\
	\& A \& B
	\arrow["p"', from=1-1, to=2-2]
	\arrow["F"', from=1-2, to=1-1]
	\arrow["{{\overline{f}}}", from=1-2, to=1-3]
	\arrow["{{\plb{f}{q}}}"', from=1-2, to=2-2]
	\arrow["\lrcorner"{anchor=center, pos=0.125}, draw=none, from=1-2, to=2-3]
	\arrow["q", from=1-3, to=2-3]
	\arrow["f"', from=2-2, to=2-3]
\end{tikzcd}}
    \caption{Diagram for $( \, f \, , \, \alpha \overset{\,\,\,F}{\leftarrow} \plb{f}{\beta} \overset{\overline{f}}{\to} \beta \, )$ being an arrow in $\mathrm{Lens}(\mathcal{L})$.}
    \label{fig:diag_lenses}
\end{figure}


Actually, we mainly need the full subcategory 
of trivial bundles, i.e.\ first projections. Concretely:

\begin{definition}\label{def:ELens}
Let $\mathcal{E}\mathrm{Lens}(\mathcal{L})$ (``$\mathcal{E}$'' stands for ``Euclidean\footnote{For the sake of a terminology, we follow here the use of ``Euclidean'' in \cite{diff_geo_book} for the $\mathbb R^n$ spaces. In this case, the category of fibre bundles over them, which has a lens structure, is Euclidean in our sense.}'') 
be the category with:
    \begin{itemize}
    \item Objects are first projections $\pi_1:\binom{A\times X}{A}$
    \item An arrow from $\pi_1:\binom{A\times X}{A}$ to $\pi_1:\binom{B\times Y}{B}$ is given by an $f:A\to B$ and a span $A\times X \overset{F}{\longleftarrow} A\times Y \overset{f\times 1}{\longrightarrow} B\times Y$ such that $F;\pi^{A,X}_1=\pi^{A,Y}_1$.
    \end{itemize}
\end{definition}

\begin{remark}
    The definition above does make sense: the span satisfies the pullback condition of $\mathrm{Lens}(\mathcal{L})$ so that the arrows above are arrows in $\mathrm{Lens}(\mathcal{L})$, and the identities and composition are inherited from $\mathrm{Lens}(\mathcal{L})$.
The only non-trivial part is to justify that the arrows above are closed w.r.t.\ composition in $\mathrm{Lens}(\mathcal{L})$: 
one can check that the composition $\pi_1:\binom{A\times X}{A} \overset{(f,F)}{\to} \pi_1:\binom{B\times Y}{B} \overset{(g,G)}{\to} \pi_1:\binom{C\times Z}{C}$ in $\mathrm{Lens}(\mathcal{L})$ of two arrows of $\mathcal{E}\mathrm{Lens}(\mathcal{L})$ is given by the pair of first component
$f;g$ and second componenet $A\times X \overset{\langle \pi_1, (f\times 1);G;\pi_2\rangle;F}{\longleftarrow} A\times Z \overset{(f;g)\times 1}{\longrightarrow} C\times Z)$, which is an arrow of $\mathcal{E}\mathrm{Lens}(\mathcal{L})$.
\end{remark}

Let us now turn to see how lenses express reverse differentiation.

\begin{proposition}\label{prop:difcat_to_lenses}
Let $\cC_!$ be a Cartesian closed differential category which is the coKleisli of a category $\cC$ as in Section \ref{sec:log_rel}.
Remember that $\cC$ comes with a bijection $\chi:\cC(D\otimes E,F) \simeq \cC(D,\ilhom{\dual F}{\dual E})$.
We use it to define the \emph{reverse differential of $f:!A\to B$ in $\cC$} as $\rho
f:=!A\otimes \dual B \overset{\uncurry\chi\partial f}{\lto} \dual A$, and the \emph{reverse differential $Rf \in \cC_!(A\times B^{\bot},A^{\bot})$ of $f$ in $\cC_!$} as $!(A\& \dual B)\simeq !A\otimes !(\dual B) \overset{1\otimes d}{\lto} !A\otimes \dual B \overset{\rho f}{\lto} \dual A$.
Then we have a functor $D:\cC_! \to \mathcal{E}\mathrm{Lens}(\cC_!)$ defined by sending $A$ to $\pi_1:\binom{A\times A^{\bot}}{A}$ and $A\overset{f}{\to}B$ to 
    $(\quad f \quad  , \quad  A\times A^{\bot} \overset{\langle \pi_1,Rf \rangle}{\longleftarrow} A\times B^{\bot} \overset{f\times 1}{\longrightarrow} B\times B^{\bot} \quad )$.
\end{proposition}



\begin{remark}\label{rmk:log_rel_from_functor}
    Remembering Figure \ref{fig:PM_to_me}, one sees that the passage from $M\overset{\lambda\partial}{\bowtie} S$ to $M\bowtie S$ contains the same information as the functor $D$ above: both build the reverse differential.
\end{remark}



        
        

We can 
also start from a reverse differential category (cfr.\ \cite[Example 28]{ReverseTC24}): by diagram chasings (in the case of composition one reasons on the diagram in Figure \ref{fig:revdifcat_to_lenses}), we have:

\begin{proposition}\label{prop:revdifcat_to_lenses}
    Let $\mathcal{L}$ be a Cartesian reverse differential category (\cite[Definition 13]{Reverse_derivative_categories}). 
    We have a functor $\mathcal{T}^*:\mathcal{L} \to \mathcal{E}\mathrm{Lens}(\mathcal{L})$ defined by sending $A$ to $\pi_1:\binom{A\times A }{A}$ and $A\overset{f}{\to}B$ to $(\, f \,  , \,  A\times A  \overset{\langle \pi_1,Rf \rangle}{\longleftarrow} A\times B  \overset{f\times 1}{\longrightarrow} B\times B  \, )$,
    where $Rf: A\times B  \to A $ in $\mathcal{L}$ is the reverse differential of $f$ (which is a primitive data in $\mathcal{L}$).
\end{proposition}

\begin{figure}[t]
    \centering
    \boxed{
    \begin{tikzcd}[ampersand replacement=\&,cramped]
	\&\& {A\times C} \\
	\& {A\times B} \&\& {B\times C} \\
	{A\times A} \&\& {B\times B} \&\& {C\times C} \\
	A \&\& B \&\& C
	\arrow["{{\langle \pi_1, (f\times 1_C);Rg \rangle}}"'{pos=0.3}, from=1-3, to=2-2]
	\arrow["{{f\times 1}}", from=1-3, to=2-4]
	\arrow["\lrcorner"{anchor=center, pos=0.125, rotate=-45}, draw=none, from=1-3, to=3-3]
	\arrow["{{\langle \pi_1, Rf \rangle}}"'{pos=0.4}, from=2-2, to=3-1]
	\arrow["{{f\times 1}}", from=2-2, to=3-3]
	\arrow["{{\pi_1}}", from=2-2, to=4-1]
	\arrow["\lrcorner"{anchor=center, pos=0.125}, shift right=2, draw=none, from=2-2, to=4-3]
	\arrow["{{\langle \pi_1, Rg \rangle}}"'{pos=0.4}, from=2-4, to=3-3]
	\arrow["{{g\times 1}}", from=2-4, to=3-5]
	\arrow["{{\pi_1}}", from=2-4, to=4-3]
	\arrow["\lrcorner"{anchor=center, pos=0.125}, shift right=2, draw=none, from=2-4, to=4-5]
	\arrow["{{\pi_1}}"', from=3-1, to=4-1]
	\arrow["{\pi_1}"', from=3-3, to=4-3]
	\arrow["{{\pi_1}}", from=3-5, to=4-5]
	\arrow["f"', from=4-1, to=4-3]
	\arrow["g"', from=4-3, to=4-5]
\end{tikzcd}}
    \caption{Diagram for the proof of Proposition \ref{prop:revdifcat_to_lenses}.}
    \label{fig:revdifcat_to_lenses}
\end{figure}


\begin{figure}[t]
    \centering
    \boxed{
    \begin{tikzcd}[ampersand replacement=\&,cramped]
	{\cotan A} \& {f^*\cotan B} \& {\cotan B} \\
	\& A \& B
	\arrow["{{p^*_A}}"', from=1-1, to=2-2]
	\arrow["{\cotan f}"', from=1-2, to=1-1]
	\arrow["{{{\overline{f}}}}", from=1-2, to=1-3]
	\arrow["{{{\plb{f}{p}^*_B}}}"', from=1-2, to=2-2]
	\arrow["\lrcorner"{anchor=center, pos=0.125}, draw=none, from=1-2, to=2-3]
	\arrow["{{p^*_B}}", from=1-3, to=2-3]
	\arrow["f"', from=2-2, to=2-3]
\end{tikzcd}}
    \caption{Diagram for defining the functor $\mathcal{T}^*$ of Proposition \ref{prop:revtangcat_to_lenses}. Remark that the functor is well defined because the left triangle commutes thanks to \protect\cite[Proposition 21.(ii).(3), left diagram]{ReverseTC24}.}
    \label{fig:rev_tang_cat}
\end{figure}

Let us now mention reverse tangent categories \cite[Definition 24]{ReverseTC24}, which provide the general geometric picture of reverse differentiation and whose canonical example is $\mathbf{SMan}$ (\cite[Example 27]{ReverseTC24}). 
Roughly speaking, such a category $\mathcal{L}$ is a tangent category, i.e.\ a differential category with non-trivial tangent bundles $p_A:\binom{\tan A}{A}$ \cite[Definition 1]{ReverseTC24}, equipped with:
a full subcategory $\mathrm{DBun}_D(\mathcal{L})$ of $\mathcal{L}$ of differential bundles which behave like cotangent bundles (\cite[Definition 16, Definition 23]{ReverseTC24});
its canonical fibration $\mathrm{DBun}_D(\mathcal{L})\to\mathcal{L}$ and dual fibration $\mathrm{DBun}_D^\circ(\mathcal{L})\to\mathcal{L}$ (\cite[Propositions 17, 21]{ReverseTC24});
an involutive fibration morphism $\mathrm{DBun}_D(\mathcal{L})\to\mathrm{DBun}_D^\circ(\mathcal{L})$ giving the dual bundle $p^*:\binom{\alpha^*}{A}$ of a differential bundle $p:\binom{\alpha}{A}$.

\begin{proposition}\label{prop:revtangcat_to_lenses}
    Let $\mathcal{L}$ be a reverse tangent category. 
    We have a functor $\mathcal{T}^*:\mathcal{L} \to \mathrm{Lens}(\mathcal{L})$ defined by sending $A$ to $p^*_A:\binom{T^*\! A }{A}$ and $A\overset{f}{\to}B$ to $(\, f \,\,  , \,\,  \cotan A  \overset{\cotan f}{\longleftarrow} f^*\cotan B \overset{\overline{f}}{\longrightarrow} \cotan B  \, )$,
    where $p^*_A$ is the dual of the tangent bundle on $A$, $\overline{f}$ is part of the pullback square in $\mathcal{L}$ of Figure \ref{fig:rev_tang_cat} and $(f,\cotan f)$ is the image of $(f,\tan f)$ under the involution of $\mathcal{L}$ (\cite[Definition 23]{ReverseTC24}), where $(f,Tf)$ is defined in \cite[Example 2.4(ii)]{diff_bun_fib_tan_cat}.
\end{proposition}
\begin{proof}
A reverse tangent category admits a reverse tangent bundle functor $\mathcal{T}^*:\mathcal{L}\to\mathrm{DBun}_D^\circ(\mathcal{L})$ as in \cite[Definition 25]{ReverseTC24}; 
one can think of $\mathrm{DBun}_D^\circ(\mathcal{L})$ as a subcategory of $\mathrm{Lens}(\mathcal{L})$ (just consider the left diagram of \cite[Proposition 21.ii(3)]{ReverseTC24}, i.e.\ ignore the differential part).
Composing with the inclusion, we get 
Proposition \ref{prop:revtangcat_to_lenses}.
\end{proof}

We will actually not work in such framework in the following, and we will clarify later why we mentioned it, in addition than for the sake of clarity.

\section{Dialectica and Lenses}\label{sec:dial_is_lens}

\paragraph*{Dialectica categories}

Dialectica categories \cite{dePaivaPHD91} are the categorical formulation of the Dialectica transformation of an implication.
Given a Cartesian closed differential category $\mathcal{L}$, in \cite[Proposition 5.7]{KerjeanPedrot24} it is defined a functor from $\mathcal{L}$ to its Dialectica category, in order to explain the link between Dialectica and Differentiation. 
The authors suggest it can be generalised.
We will see that the functor does not really use the fact that we use a Dialectica category, in that it does not use subobjects, and thus 
immediately lifts to lenses (Corollary \ref{cor:no_subsets}), 
and we can then generalise it to reverse differential categories (Proposition \ref{prop:revdifcat_to_lenses}), ignoring the closedness condition. 


We denote a subobject $a$ of an object $A$ in a category by the abuse of notation $a \overset{a}{\rightarrowtail} A$.
We thus mean the mono $a$ to $A$ and its equivalence class.

\begin{definition}\label{def:dial_cat}
    The Dialectica Category \cite{dePaivaPHD91} $\mathrm{Dial}(\mathcal{L})$ over $\mathcal{L}$ is made of:

Objects are the data of two objects $A,X$ in $\mathcal{L}$ and a subobject $a \rightarrowtail A\times X$ in $\mathcal{L}$ (in $\mathbf{Set}$, $a$ is a subset of $A\times X$, playing the role of a formula with two free variables, e.g.\ a binary predicate).
An arrow from $(A,X,a)$ to $(B,Y,b)$ is the data of an $f:A\to B$ and a $F: A\times Y \to X$ in $\mathcal{L}$ such that given the diagram of pullbacks in Figure \ref{fig:dial_cat}, there exists exactly one dotted arrow as in the figure making the triangle commute.
\end{definition}

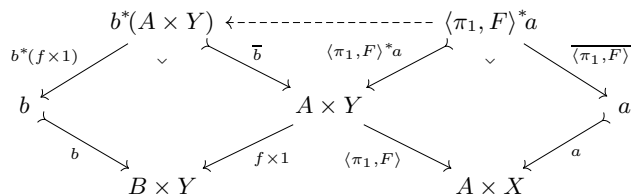
\begin{figure}[t]
    \centering
    \boxed{
    \begin{tikzcd}[ampersand replacement=\&,cramped]
	\& {\plb{b}{(A\times Y)}} \&\& {\plb{\langle \pi_1,F \rangle}{a}} \\
	b \&\& {A\times Y} \&\& a \\
	\& {B\times Y} \&\& {A\times X}
	\arrow["{\plb{b}{(f\times 1})}"', from=1-2, to=2-1]
	\arrow["{\overline b}", tail , from=1-2, to=2-3]
	\arrow["\lrcorner"{anchor=center, pos=0.125, rotate=-45}, draw=none, from=1-2, to=3-2]
	\arrow[dashed, from=1-4, to=1-2]
	\arrow["{\plb{\langle \pi_1,F \rangle}{a}}"', tail, from=1-4, to=2-3]
	\arrow["{\overline{\langle \pi_1,F \rangle}}", from=1-4, to=2-5]
	\arrow["\lrcorner"{anchor=center, pos=0.125, rotate=-45}, draw=none, from=1-4, to=3-4]
	\arrow["b"', tail, from=2-1, to=3-2]
	\arrow["{f\times 1}", from=2-3, to=3-2]
	\arrow["{\langle \pi_1,F \rangle}"', from=2-3, to=3-4]
	\arrow["a", tail, from=2-5, to=3-4]
\end{tikzcd}}
    \caption{Diagram for $( \, f:A\to B \, , \, F: A\times Y \to X \, )$ being an arrows of $\mathrm{Dial}(\mathcal{L})$. In $\mathbf{Set}$, this reads as: $(f(x),y)\in b$ for all $(x,y)\in A\times Y$ such that $(x,F(x,y))\in a$.}
    \label{fig:dial_cat}
\end{figure}

\begin{remark}
    In a setting where tangent and cotangent spaces are isomorphic, the typing of $F$ in the previous definition is precisely that of the reverse differential of $f$.
    Moreover, the identity on $(A,X,a)$ in $\mathrm{Dial}(\mathcal{L})$ is $(1_A,\pi_2^{A,X})$, and one can check that the composition $(f,F);(g,G)$, given in \cite[Proposition 1]{dePaivaPHD91}, coincides with $(f;g, \langle \pi_1, (f\times 1);G \rangle;F)$, which is the same as for the composition of reverse differentials.
\end{remark}

Contrarily to the typing of $F$ and its composition, the condition involving subobjects in the definition of an arrow of $\mathrm{Dial}(\mathcal{L})$ is not immediately clear in geometric terms.
This is precisely the phenomenon that we discussed in Figure \ref{fig:tom}.
In fact, we will get rid of this condition the following (e.g.\ Proposition \ref{prop:difcat_to_lenses}, Corollary \ref{cor:no_subsets}), as it appears not necessary in order to link Differentiation and Dialectica. 
Let us start with the following easy:


\begin{proposition}\label{prop:G}
We have a functor $G:\mathcal{E}\mathrm{Lens}(\mathcal{L}) \to \mathrm{Dial}(\mathcal{L})$
defined as follows:
\begin{itemize}
        \item An object $\pi_1:\binom{A\times X}{A}$ is sent to $(A,X,1_{A\times X})$;
        
        \item An arrow $( \quad f:A\to B \quad , \quad A\times X \overset{F}{\longleftarrow} A\times Y \overset{f\times 1}{\longrightarrow} B\times Y \quad )$ from $\pi_1:\binom{A\times X}{A}$ to $\pi_1:\binom{B\times Y}{B}$ is sent to $(f , \, F;\pi_2)$ from $(A,X,1_{A\times X})$ to $(B,Y,1_{B\times Y})$.
\end{itemize}
\end{proposition}

The functor above only uses full subobjects in $\mathrm{Dial}(\mathcal{L})$, so only a strict subcategory of it:

\begin{proposition}\label{prop:Giso}
The image of the functor $G$ is the following full subcategory $\mathcal{E}\mathrm{Dial}(\mathcal{L})$ of $\mathrm{Dial}(\mathcal{L})$, and $G:\mathcal{E}\mathrm{Lens}(\mathcal{L}) \to \mathcal{E}\mathrm{Dial}(\mathcal{L})$ is an isomorphism.

Objects are given by full subobjects $(A,X,1_{A\times X})$.
An arrow from $(A,X,1_{A\times X})$ to $(B,Y,1_{B\times Y})$ is given by arrows $f:A\to B$ and $F:A\times Y \to X$.

The inverse $G^{-1}:\mathcal{E}\mathrm{Dial}(\mathcal{L}) \to \mathcal{E}\mathrm{Lens}(\mathcal{L})$ of $G$ is given as follows:
An object $(A,X,1_{A\times X})$ is sent to $\pi_1^{A,X}$.  
An arrow $(f:A\to B,\, F: A\times Y \to X)$ from $(A,X,1_{A\times X})$ to $(B,Y,1_{B\times Y})$ is sent to $( \, f \, , \, A\times X \overset{\langle \pi_1, F\rangle}{\longleftarrow} A\times Y \overset{f\times 1}{\longrightarrow} B\times Y \, )$ from $\pi_1^{A,X}$ to $\pi_1^{B,Y}$.
\end{proposition}
\begin{proof}
That $\mathcal{E}\mathrm{Dial}(\mathcal{L})$ is a subcategory of $\mathrm{Dial}(\mathcal{L})$ is immediate to check.
In order to show that $G^{-1}((f,F);(g,G))=G^{-1}(f,F);G^{-1}(g,G)$ one uses the fact that $\langle \pi_1^{A,Z} , (f\times 1_Z);G) \rangle ; \langle \pi_1^{A,Y} , F \rangle =$ $\langle \pi_1^{A,Z} , \langle \pi_1^{A,Z} , (f\times 1_Z);G \rangle ; F \rangle$, which can be immediately checked.
In order to see that $(G;G^{-1})(f,F)=(f,F)$, one uses that $F;\pi_1=\pi_1$, which is given by definition of $\mathrm{Lens}(\mathcal{L})$.
\end{proof}

The previous results could be rephrased via dependent types in a more syntactic fashion.
We use them now to discuss the already mentioned functor $\cC_!\to \mathrm{Dial}(\cC_!)$ defined in \cite[Proposition 5.7]{KerjeanPedrot24} in order to suggest a link between Dialectica categories and Differentiation, for $\cC_!$ a differential category as in Section \ref{sec:log_rel}.
In the very last lines of \cite{KerjeanPedrot24}, the authors also suggest a similar functor $\mathcal{L}\to\mathrm{Dial}(\mathcal{L})$ for $\mathcal{L}$ a reverse differential category.

\begin{corollary}\label{cor:no_subsets}
Remember the functors $D$ and $\mathcal{T}^*$ from Proposition \ref{prop:difcat_to_lenses}, \ref{prop:revdifcat_to_lenses}.
The functor in \cite[Proposition 5.7]{KerjeanPedrot24} is actually the composition $\cC_! \overset{D}{\to} \mathcal{E}\mathrm{Lens}(\cC_!) \overset{G}{\simeq} \mathcal{E}\mathrm{Dial}(\cC_!) \xhookrightarrow{} \mathrm{Dial}(\cC_!)$.
The one at the very last lines of \cite{KerjeanPedrot24} is actually the composition $\mathcal L \overset{\mathcal T^*}{\to} \mathcal{E}\mathrm{Lens}(\mathcal L) \overset{G}{\simeq} \mathcal{E}\mathrm{Dial}(\mathcal L) \xhookrightarrow{} \mathrm{Dial}(\mathcal L)$.
\end{corollary}

This tells us something interesting:
what relates Dialectica to differentiation is better understood in terms of lenses, rather than of Dialectica categories. 
This was not remarked in \cite{KerjeanPedrot24}, and we will explicit it even more in the reminder of the paper.
It also tells us that the natural generalization to non-trivial subobjects is the one in Proposition \ref{prop:revtangcat_to_lenses} with reverse tangent categories (and this is also why we mentioned them).

\begin{remark}
The lens structure involved in the corollary above is merely the Euclidean one, so only trivial bundles.
From the logical point of view, it means that we do not look at formulas, and this may seem strange.
But we can understand it by looking at the computational formulation of Dialectica in $\Pdial$ \cite[Section 8.3.2 and 9.1.4]{pedrot_thesis}: the subobjects (i.e.\ the formulas) are not there anymore, because their role (which is that of an orthogonality relation) is in a sense already encoded by the witnesses ($W$) and counters ($C$). 
\end{remark}

\paragraph*{Dialectica is a functor to Lenses}

We understood that Dialectica is related with differentiation thanks to its relation to Euclidean lenses.
One can then wonder if it is possible to directly express it (in its structural version, as in $\Pdial$) as a transformation involving them.
After all, for now Dialectica was just the syntactic transformation of Figure \ref{fig:Pdial}.
This is indeed possible and it is of high relevance, because it synthetizes, in a sense, the whole claim of this paper.

\begin{figure}[t]
    \centering
\[\boxed{\begin{tikzcd}[ampersand replacement=\&,cramped]
	{W(A)\times \pmon{C(A)}} \&\& {W(A)\times \pmon{C(B)}} \&\& {W(B)\times \pmon{C(B)}} \\
	\&\& {W(A)} \&\& {W(B)}
	\arrow["{z^1}"', from=1-1, to=2-3]
	\arrow["{\langle z^1{,}\,z^2\bind M_{(z^1)} \rangle}"', from=1-3, to=1-1]
	\arrow["{\langle \wdial M {,} \, z^2 \rangle}", from=1-3, to=1-5]
	\arrow["{z^1}"', from=1-3, to=2-3]
	\arrow["{z^1}", from=1-5, to=2-5]
	\arrow[""{name=0, anchor=center, inner sep=0}, "{\wdial M}"', from=2-3, to=2-5]
	\arrow["\lrcorner"{anchor=center, pos=0.125}, draw=none, from=1-3, to=0]
\end{tikzcd}}\]
    \caption{One can check that, in $\Pdial_{\mathrm{cat}}$, the square is a pullback and the triangle commutes.}
    \label{fig:LensPcat}
\end{figure}

Let $\stlc_{\mathrm{cat}}$ and $\Pdial_{\mathrm{cat}}$ be the syntactic categories induced by the simply-typed $\lambda$-calculus and $\Pdial$ as usual, i.e.: objects are types, an arrow from $A$ to $B$ is the equational semantics class of a term $z:A\vdash M:B$, the identities are variables $z:A\vdash z:A$ and composition is substitution.
Now, $\Pdial_{\mathrm{cat}}$ does not have pullbacks in general, but we can still define the category $\mathcal{E}\mathrm{Lens}(\Pdial_{\mathrm{cat}})$ over it,
exactly as in Definition \ref{def:ELens} (ignoring that it is a subcategory of a whole category of lenses).

Now by looking at Figure \ref{fig:LensPcat}, we can prove the following (remember that the $\Pdial$-term $N^i$ stands for the projection of the $\Pdial$-term $N$):

\begin{proposition}\label{prop:dial_to_ElensP}
    We have a functor $\stlc_{\mathrm{cat}}\to\mathcal{E}\mathrm{Lens}(\Pdial_{\mathrm{cat}})$ defined as follows:
    \begin{itemize}
        \item An object $A$ is sent to $(z:W(A)\times \pmon{C(A)}\vdashPdial z^1:W(A))$;
        
        \item An arrow $(z:A\vdash_\stlc M:B)$ in $\stlc_{\mathrm{cat}}$ from $A$ to $B$ is sent to the arrow in $\mathcal{E}\mathrm{Lens}(\Pdial_{\mathrm{cat}})$ from $(z:W(A)\times \pmon{C(A)}\vdashPdial z^1:W(A))$ to $(z:W(B)\times \pmon{C(B)}\vdashPdial z^1:W(B))$ given by $(z:W(A)\vdashPdial \wdial M : W(B))$ and the span: 
        \[W(A)\times \pmon{C(A)} \overset{{\langle z^1{,}\,z^2\bind M_{(z^1)} \rangle}}{\longleftarrow} W(A)\times \pmon{C(B)} \overset{{\langle \wdial M {,} \, z^2 \rangle}}{\longrightarrow} W(B)\times \pmon{C(B)}.\]
\end{itemize}
\end{proposition}


\begin{corollary}\label{cor:dial=functor}
    We have a functor $\stlc_{\mathrm{cat}}~\to~\mathcal{E}\mathrm{Lens}(\Pdial_{\mathrm{cat}})~\overset{G}{\to}~\mathcal{E}\mathrm{Dial}(\Pdial_{\mathrm{cat}})$ sending a simple type $A$ to $(W(A),\pmon{C(A)},1)$ and a term
    $(x:A\vdash_\stlc M:B)$ from $A$ to $B$ to the arrow from $(W(A),\pmon{C(A)},1)$ to $(W(B),\pmon{C(B)},1)$ given by the pair
    $(x:W(A)\vdashPdial \wdial M : W(B))$ and $(z:W(A)\times \pmon{C(B)}\vdashPdial z^2\bind M_{(z^1)}:\pmon{C(A)})$.
    For $z:=\langle x,[y]\rangle$, since $y\notin \cdial{M}{x}$, the latter term becomes $x:W(A)\vdashPdial M_{x}:C(B)\to\pmon{C(A)}$. 
\end{corollary}

Corollary~\ref{cor:dial=functor} \emph{precisely} expresses the Dialectica transformation of $\stlc$, together with its soundness Theorem, as a functor. 
To the best of our knowledge, expressing the Dialectica $\stlc\to\Pdial$ of \cite{KerjeanPedrot24,Pedrot14}
in this way was not remarked before and, 
even if straightforward, we think that it is very instructive.
Most importantly, it factors it through lenses, showing that the actual transformation 
is performed \emph{solely} by the functor of Proposition \ref{prop:dial_to_ElensP}, the map $G$ being unrelated with Dialectica.
Its compositional properties follow then from this factorization, and ultimately make it behave as a reverse differentiation transformation.

\section{Future work questions}\label{sec:conclusions_futurework}


    The first question is to express 
    Definition \ref{def:log_rels} in a \emph{reverse} differential category. 
    It should be possible, 
    but it should be Cartesian closed in order to interpret $\lambda$-calculus, and this has not been explored in the literature yet. 
    Its syntactic counterpart asks for defining a ``\emph{reverse} differential $\lambda$-calculus''. 
    In a sense, it already appears 
    in 
    \cite{ong_mak}, where the authors mimic pullbacks of differential 1-forms.
    A natural goal is then to see if this is suitable for expressing Dialectica, even if its involved operational semantics makes it difficult to work with. 
    
    Another interesting question is that of lifting the relations in Definition \ref{def:log_rels} to a dependently typed language (or different logical systems). 
    The natural starting points would be \cite{moss_glenn_dialTT18,nunes_vakar_grot_dial2024,pedrot_thesis}, where it is shown how to formulate Dialectica for dependent types.
    However, its categorical counterpart should be that of Cartesian \emph{closed} reverse tangent categories, which again have not been explored in the literature yet.
    
    Finally, remark that models of differential linear logic on the lines of \cite{EhrhardDiLL} (such as the ones that we used in Section \ref{sec:log_rel}), do \emph{not} include geometric models like $\mathbf{SMan}$ in the Cartesian closed case.
    The recent setting of \emph{linearly closed reverse differential categories} in \cite{Jacobians_Gradients_for_CDC} precisely allows to keep geometric examples while still being able to (un)curry linear maps.
    Would that be the correct general framework in order to formulate the present work? 

\bibliographystyle{acm}
\bibliography{./mybib}

\medskip

\newpage

\appendix

\section{Appendix: a few of proofs}

\begin{proof}[Proof of Proposition \ref{prop:Pm_to_me}]
    By straightforward mutual induction on $B$.
    We do not give the details because it involves the relation in \cite{KerjeanPedrot24} which we did not report here. The requirement that $\llbracket\_\rrbracket$ be full complete is used in the case of $(3)$ and $(4)$.
\end{proof}

For Section \ref{sec:diff_lenses}, we have the following.

\begin{figure}[t]
    \centering
    \boxed{
    \begin{tabular}{lcr}
    \begin{tikzcd}[ampersand replacement=\&,cramped]
	\& P \\
	\&\& {\plb{\overline{f}}{\plb g \gamma}} \\
	\& {\plb f \beta} \&\& {\plb g \gamma} \\
	\alpha \&\& \beta \&\& \gamma \\
	A \&\& B \&\& C
	\arrow[color={rgb,255:red,214;green,92;blue,92}, squiggly, from=1-2, to=2-3]
	\arrow[color={rgb,255:red,214;green,92;blue,92}, curve={height=6pt}, dashed, from=1-2, to=3-2]
	\arrow[color={rgb,255:red,214;green,92;blue,92}, curve={height=30pt}, dotted, from=1-2, to=3-4]
	\arrow[color={rgb,255:red,214;green,153;blue,92}, curve={height=-30pt}, from=1-2, to=4-5]
	\arrow[color={rgb,255:red,214;green,153;blue,92}, curve={height=30pt}, from=1-2, to=5-1]
	\arrow["{\plb{\overline{f}}{G}}"'{pos=0.3}, from=2-3, to=3-2]
	\arrow["{\overline{{\overline{f}}}}", from=2-3, to=3-4]
	\arrow["\lrcorner"{anchor=center, pos=0.125, rotate=-45}, draw=none, from=2-3, to=4-3]
	\arrow["F"'{pos=0.4}, from=3-2, to=4-1]
	\arrow["{{{\overline{f}}}}", color={rgb,255:red,51;green,51;blue,255}, from=3-2, to=4-3]
	\arrow["{\plb f q}", color={rgb,255:red,51;green,51;blue,255}, from=3-2, to=5-1]
	\arrow["\lrcorner"{anchor=center, pos=0.125}, shift right=2, color={rgb,255:red,51;green,51;blue,255}, draw=none, from=3-2, to=5-3]
	\arrow["G"'{pos=0.4}, from=3-4, to=4-3]
	\arrow["{{{\overline{g}}}}", color={rgb,255:red,153;green,92;blue,214}, from=3-4, to=4-5]
	\arrow["{\plb q r}", color={rgb,255:red,153;green,92;blue,214}, from=3-4, to=5-3]
	\arrow["\lrcorner"{anchor=center, pos=0.125}, shift right=2, color={rgb,255:red,153;green,92;blue,214}, draw=none, from=3-4, to=5-5]
	\arrow["p"', from=4-1, to=5-1]
	\arrow["q"', color={rgb,255:red,51;green,51;blue,255}, from=4-3, to=5-3]
	\arrow["r", color={rgb,255:red,153;green,92;blue,214}, from=4-5, to=5-5]
	\arrow["f"', color={rgb,255:red,51;green,51;blue,255}, from=5-1, to=5-3]
	\arrow["g"', color={rgb,255:red,153;green,92;blue,214}, from=5-3, to=5-5]
\end{tikzcd} & & 
    \begin{tikzcd}[ampersand replacement=\&,cramped]
	{\plb{f}{\beta}=\plb{\overline f}{\beta}} \& \beta \\
	{\plb{f}{\beta}} \& \beta \\
	A \& B
	\arrow["{{\overline{\overline{f}}=\overline{f}}}", from=1-1, to=1-2]
	\arrow["{{1={\plb{\overline{f}}{1}}}}"', from=1-1, to=2-1]
	\arrow["\lrcorner"{anchor=center, pos=0.125}, draw=none, from=1-1, to=2-2]
	\arrow["1", from=1-2, to=2-2]
	\arrow["{{\overline{f}}}"{description}, from=2-1, to=2-2]
	\arrow["{\plb{f}{p}}"', from=2-1, to=3-1]
	\arrow["\lrcorner"{anchor=center, pos=0.125}, draw=none, from=2-1, to=3-2]
	\arrow["p", from=2-2, to=3-2]
	\arrow["f"', from=3-1, to=3-2]
\end{tikzcd}
    \end{tabular}}
    \caption{$\mathrm{Lens}(\mathcal{L})$ is a category. Left: diagram for composition; right: for identities.}
    \label{fig:lens_cat}
\end{figure}

\begin{proposition}\label{prop:lens_well_def}
    Definition \ref{def:lenses} makes sense, i.e.\ $\mathrm{Lens}(\mathcal{L})$ is a well-defined category.
\end{proposition}
\begin{proof}
The only non trivial part is that our composition gives indeed an arrow of our claimed category, and that our claimed identity is indeed such.
In both cases, the argument is similar to the pasting law for pullbacks.
For the composition, one sees that our definition claims to take the composition $p:\binom{\alpha}{A} \overset{(f,F)}{\to} q:\binom{\beta}{B} \overset{(g,G)}{\to} r:\binom{\gamma}{C}$ to be $(f;g, \, \alpha \overset{\plb{\overline{f}}{G};F}{\longleftarrow} \plb{\overline{f}}{\plb g \gamma} \overset{\overline{{\overline{f}}};\overline{g}}{\longrightarrow} \gamma)$. 
To show that this is indeed an arrow in our claimed category, we consider the left diagram in Figure \ref{fig:lens_cat},
which is read as follows:
given the blue and purple pullbacks, and given the composition span (which is defined via the black pullback), in order to show that our composition is well defined, we show both that $A \overset{\plb{\overline{f}}{G};\plb f q}{\longleftarrow} \plb{\overline{f}}{\plb g \gamma} \overset{\overline{{\overline{f}}};\overline{g}}{\longrightarrow} \gamma$ is the pullback of $A \overset{f;g}{\longrightarrow} C \overset{r}{\leftarrow} \gamma$ and that $\plb{\overline{f}}{G};F;p=\plb{\overline{f}}{G};\plb f q$.
The latter is immediate (because $(f,F)$ is an arrow).
For the former, the commutation is immediate (because $(g,G)$ is an arrow), and given the two orange arrows, one obtains the unique squiggly red arrow by first obtaining the unique dotted red arrow (using the purple pullback), then obtaining the unique dotted arrow (using the blue pullback), and finally the desired one (using the black pullback).
For identities, one uses the fact that, because $\mathrm{Span}(\mathcal{L})$ is a category, the upper square of the right diagram of Figure \ref{fig:lens_cat} is a pullback as soon as the bottom one is.
\end{proof}

In the discussion just after Definition \ref{def:ELens} justifying that it makes sense, the mentioned composition immediately follows from the following: 

\begin{lemma}\label{lm:tec_lm}
    In a category with products, given a commutative square and a triangle as the ones in purple in Figure \ref{fig:tec_lemma}, the pullback of the diagram $A\times Y \overset{f\times 1}{\longrightarrow} B\times Y \overset{G}{\longleftarrow} B\times Z$ is the black one in the same figure.
\end{lemma}
\begin{proof}
    The only non trivial part is to show that, given a $P$ and the orange arrows $h_1,h_2$ under the commutation hypotheses $h_1;(f\times 1)=h_2;G$, one has that the red dotted arrow makes the two red+black=orange triangles commute.
    
    Let us call $h$ the red arrow.
    We show that both $h; \langle \pi_1 \, , \, (f\times 1);G;\pi_2 \rangle$ and $h_1$ satisfy the universal property of $\langle h_1;\pi_1 \, , \, h;(f\times 1);G;\pi_2 \rangle$ (so the left red+black=orange triangle commutes), and both $h;(f\times 1)$ and $h_2$ satisfy the universal property of $\langle h_1;(f\times 1);\pi_1 \, , \, h_2;\pi_2 \rangle$ (so the right red+black=orange triangle commutes).
   
   We first show the right red+black=orange triangle.
    For $h;(f\times 1)$, we trivially have $h;(f\times 1);\pi_1=h_1;(f\times 1);\pi_1$, and $h;(f\times 1);\pi_2=h;\pi_2=h_2;\pi_2$ by definition of $\times$ and of $h$.
    For $h_2$, we trivially have $h_2;\pi_2=h_2;\pi_2$, and $h_2;\pi_1=h_1;(f\times 1);\pi_1$ by using the commutative triangle of the hypothesis and then the commutation hypotheses on $h_1,h_2$.

    Now we can prove the left red+black=orange triangle.
   For $h_1$, we trivially have $h_1;\pi_1=h_1;\pi_1$, and $h_1;\pi_2=h;(f\times 1);G;\pi_2$ using the fact that $\pi_2^{A,Y}=(f\times 1);\pi_2^{B,Y}$ by definition of $\times$, then the commutation hypotheses on $h_1,h_2$ and finally the just showed red+black=orange triangle involving $h_2$.
    For $h; \langle \pi_1 \, , \, (f\times 1);G;\pi_2 \rangle$, we immediately have $h; \langle \pi_1 \, , \, (f\times 1);G;\pi_2 \rangle;\pi_1=h_1;\pi_1$ by definition of $\langle \cdot,\cdot \rangle$, and $h; \langle \pi_1 \, , \, (f\times 1);G;\pi_2 \rangle;\pi_2^{A,Y}=h;\langle \pi_1 \, , \, (f\times 1);G;\pi_2 \rangle;(f\times 1);\pi_2^{B,Y}=h;(f\times 1);G;\pi_2$.
\end{proof}

\begin{figure}[t]
    \centering
    \boxed{
    \begin{tikzcd}[ampersand replacement=\&]
	\&\& P \\
	\\
	\&\& {A\times Z} \\
	\& {A\times Y} \&\& {B\times Z} \\
	\&\& {B\times Y} \\
	A \&\& B
	\arrow["{{\langle h_1;\pi_1, h_2;\pi_2 \rangle}}"{description}, color={rgb,255:red,214;green,92;blue,92}, dashed, from=1-3, to=3-3]
	\arrow["{{h_1}}"', color={rgb,255:red,214;green,153;blue,92}, curve={height=30pt}, from=1-3, to=4-2]
	\arrow["{{h_2}}", color={rgb,255:red,214;green,153;blue,92}, curve={height=-30pt}, from=1-3, to=4-4]
	\arrow["{{{{\langle \pi_1, (f\times 1);G;\pi_2 \rangle}}}}"'{pos=0.3}, from=3-3, to=4-2]
	\arrow["{{f\times 1}}", from=3-3, to=4-4]
	\arrow["\lrcorner"{anchor=center, pos=0.125, rotate=-45}, draw=none, from=3-3, to=5-3]
	\arrow["{{f\times 1}}", color={rgb,255:red,153;green,92;blue,214}, from=4-2, to=5-3]
	\arrow["{{{{\pi_1}}}}", color={rgb,255:red,153;green,92;blue,214}, from=4-2, to=6-1]
	\arrow["G"'{pos=0.4}, color={rgb,255:red,153;green,92;blue,214}, from=4-4, to=5-3]
	\arrow["{{{{\pi_1}}}}", color={rgb,255:red,153;green,92;blue,214}, from=4-4, to=6-3]
	\arrow["{{{\pi_1}}}"', color={rgb,255:red,153;green,92;blue,214}, from=5-3, to=6-3]
	\arrow["f"', color={rgb,255:red,153;green,92;blue,214}, from=6-1, to=6-3]
\end{tikzcd}}
    \caption{Figure of Lemma \ref{lm:tec_lm}.}
    \label{fig:tec_lemma}
\end{figure}

\end{document}